\documentclass{amsart}

\usepackage{amssymb, amsmath, fancyhdr, hyperref, mathrsfs,a4wide,color, MnSymbol}
\usepackage[2cell,arrow,curve,matrix]{xy}

\newtheorem{Pro}{Proposition}[subsection]
\newtheorem{Le}[Pro]{Lemma}
\newtheorem{Th}[Pro]{Theorem}
\newtheorem{Rem}[Pro]{Remark}
\newtheorem{Co}[Pro]{Corollary}
\newtheorem{De}[Pro]{Definition}
\newtheorem{Exm}[Pro]{Example}

\def\ta{{\mathscr T}}
\def\set{\mathscr{S}}

\def\sc{{\sf Sets}}

\def\sF{\sf{F}}
\def\id{{\sf id}}
\def\Top{{\sf Top}}
\def\pts{\sf{Pts}}

\def\red{{\sf red}}

\def\cat{\sf Cat}
\def\Proj{{\sf Proj}}
\def\v{{\sf v}}

\def\Pts{{\bf Pts }}

\def\I{{\mathbb I}}

\def\N{{\mathbb N}}
\def\Z{{\mathbb Z}}

\def\p{\mathfrak p}
\def\q{{\mathfrak q}}
\def\n{{\mathfrak n}}
\def\m{{\mathfrak m}}
\def\Qc{{\mathfrak Qc}}
\def\f{\mathfrak{f}}
\def\fx{{\frak x}}
\def\r{\frak r}
\def\G{\mathfrak{G}}

\def\O{{\mathcal O}}
\def\F{{\mathcal F}}
\def\cX{\mathcal{P}}

\DeclareMathOperator\Hom{{\mathsf Hom}}
\DeclareMathOperator\Spec{\mathsf{Spec}}
\DeclareMathOperator\Aut{{\mathsf{Aut}}}
\DeclareMathOperator\End{{\mathsf{End}}}

\DeclareMathOperator\im{{\mathsf{Im}}}
\DeclareMathOperator{\Mg}{M^{\mathsf Gr}}

\begin{document}

\title{Topos points of quasi-coherent sheaves of monoid schemes}

\author[I. Pirashvili]{Ilia Pirashvili}
\address{	Institute for Mathematics, 	University of Osnabrück, Albrechtstr. 28a, D-\textup{49069}, Germany} 
\email{\href{mailto:ilia\_p@ymail.com}{ilia\_p@ymail.com} \vspace{2em} \href{mailto:ilia.pirashvili@uni-osnabrueck.de}{ilia.pirashvili@uni-osnabrueck.de} }

\maketitle

\begin{abstract} Let $X$ be a monoid scheme. We will show that the stalk at any point of $X$ defines a point of the topos $\Qc(X)$ of quasi-coherent sheaves over $X$. As it turns out, every topos point of $\Qc(X)$ is of this form if $X$ satisfies some finiteness conditions. In particular, it suffices for $M/M^\times$ to be finitely generated when $X$ is affine, where $M^\times$ is the group of invertible elements.
	
	This allows us to prove that two quasi-projective monoid schemes $X$ and $Y$ are isomorphic if and only if $\Qc(X)$ and $\Qc(Y)$ are equivalent.
	
	The finiteness conditions are essential, as one can already conclude by the work of A. Connes and C. Consani \cite{cc1}. We will study the topos points of free commutative monoids and show that already for $\mathbb{N}^\infty$, there are `hidden' points. That is to say, there are topos points which are not coming from prime ideals. This observation reveals that there might be a more interesting `geometry of monoids'.
\end{abstract}

\section{Introduction} 

Monoid schemes were introduced by Deitmar \cite{deitmar} after a pioneering work of Kato \cite{kato}. One of their most important aspects is that they allow us to generalise toric varieties, but they have many other applications as well. They are used in algebraic $K$-theory \cite{chww}, are a fundamental part in various models of algebraic geometry over a "field with one element" \cite{Lorscheid}, \cite{cc}, \cite{deitmar}, and have strong uses in areas such as tropical- \cite{pin} and logarithmic geometry \cite{acgh}, \cite{kato2}. We investigated vector bundles over monoid schemes in \cite{p2}. In the present paper, we shall consider the  quasi-coherent sheaves  over a monoid scheme $X$. These objects have already been considered in \cite{Lorscheid}, where some basic facts about such sheaves were established. 

Our interest is to study the global properties of quasi-coherent sheaves over monoid schemes. To explain our results, let us recall a classical result of algebraic geometry: It was established by Gabriel in \cite{ab}, that the category of quasi-coherent sheaves over a noetherian scheme (in the classical sense) is a locally noetherian abelian category. Further, the set of isomorphism classes of indecomposable injective objects of said category, is in a one to one correspondence with the underlying topological space of a given scheme.

In this paper, we consider the category $\Qc(X)$ of quasi-coherent sheaves over a monoid scheme $X$. If $X=\Spec(M)$ is an affine monoid scheme, then $\Qc(\Spec(M))$ is equivalent to the category $\set_M$ of $M$-sets. This is naturally a topos (throughout this paper, a topos is assumed to be a Grothendieck topos). Indeed, it is one of the main examples of a topos. Theorem \ref{432.05.03} says that the category $\Qc(X)$ is also a topos. Moreover, it is a so called locally $s$-noetherian topos if $X$ is an $s$-noetherian monoid scheme. We refer to the main body of this paper for the definition of an $s$-noetherian monoid scheme, respectively topos.

There exists a general notion of a point in topos theory, called a topos point. The category of all topos points of a topos $\ta$ will be denoted by $\Pts(\ta)$. They are a core construction for any topos as the name would indicate. The category $\Pts(\set_G)$ for a group $G$, as an example, is equivalent to the category $\G$, having a single object $x$, with $\Aut(x)=\End(x)=G$, see Example \ref{t-233}. This has been known for a long time.

We will study the topos points of the topos $\Qc(X)$ of quasi-coherent sheaves over an $s$-noetherian monoid scheme $X$. As the recent work of Connes and Consani \cite{cc1} demonstrates, this problem is quite interesting even in the affine case.

Let $x$ be a point of a monoid scheme $X$. The stalk at $x$ can be considered as a functor
$${\sf Stalk}(x):\Qc(X)\rightarrow \sc.$$
It is straightforward to see that this is the inverse image part of a point of the topos $\Qc(X)$. The question we aim to answer is under what conditions on $X$ every topos point of $\Qc(X)$ is induced by a stalk of $X$. In general, there are `hidden' points (see Section \ref{connes}) even in the affine case. Under some finiteness assumptions on $X$ however, the only points of the topos $\Qc(X)$ correspond to the stalks of $X$. This is one of our main result, which is proven in Section \ref{poaff} and Section \ref{poproj}, in the affine and quasi-projective case respectively.

In addition to the conceptual nature of this result, it also significantly simplifies the calculation of $\Pts(\set_M)$ for finitely generated $M$, as it reduces it to calculating the set of prime ideals $\Spec(M)$. This, on the other hand, is very simple if the monoid $M$ is given in terms of generators and relations, using the results obtained in \cite{p1}. 

It will follow from our results that, under some conditions, two monoid schemes $X$ and $Y$ are isomorphic if and only if their associated topoi of quasi-coherent sheaves $\Qc(X)$ and $\Qc(Y)$ are equivalent, see Theorem \ref{7550804}.
\newline

The paper is organised as follows: We recall some basic facts about topoi and introduce the notion of a locally $s$-noetherian topos in Section \ref{010804}. The main result says that the gluing of locally $s$-noetherian topoi is again a locally $s$-noetherian topos. This fact is a direct analogue of a result by Gabriel \cite{ab}, which says that the gluing of locally noetherian abelian categories is again a locally noetherian abelian category.

Section \ref{020804} deals with monoids and monoid actions (on sets). We also discuss $s$-noetherian monoids. This is a class of monoids that contains finitely generated monoids, as well as all groups. We will also look at their relation with $s$-noetherian topoi. Much of the material discussed here is probably well known.

We introduce the notion of an $s$-noetherian monoid scheme in Section \ref{030804} and show that the category of quasi-coherent sheaves over such a monoid scheme is a locally $s$-noetherian topos.

In Section \ref{poaff}, we study the topos points of the category of $M$-sets, where $M$ is a monoid. We relate such points with the prime ideals of $M$ in the case when $M$ is commutative. We prove one of our main results of this paper, which says that there is a one to one correspondence between prime ideals and the isomorphism classes of points of the topos of $M$-sets, assuming $M$ is almost finitely generated. The latter term means that $M/M^\times$ is finitely generated, where $M^\times$ denotes the group of invertible elements of $M$.

We proceed in Section \ref{060804} by studying positively graded monoids and their actions on graded sets. This is a topos, much like the topos of $M$-sets and we will once again be interested in its points. This section serves as a technical preparation for the next. In Section \ref{poproj}, we will prove our main results concerning the points of the topos of quasi-coherent sheaves over a monoid scheme. In particular, we will show that the topos of quasi-coherent sheaves is rich enough to distinguish between the underlying monoid schemes in Theorem \ref{7550804}, under some assumptions on $X$.

We return to the affine case in the last section. The prime ideals are no longer going to be enough to classify all the topos points of $M$-sets. We will focus on the case when $M$ is a free monoid, generated by an infinite set $\cX$. We will also reprove some results of \cite{cc1} for countable $\cX$. Our methods and descriptions, however, are completely different.

\section{Gluing of locally $s$-noetherian topoi}\label{010804}

We will study an obvious topos-theoretical analogue of locally noetherian abelian categories in this section. The results that we obtain echo this similarity, and can be directly compared to the ones obtained by Gabriel in \cite{ab} for abelian categories. We will need this technique to prove Theorem \ref{432.05.03} below. It states that the category of quasi-coherent sheaves over a monoid scheme $X$, satisfying some finiteness conditions, is a topos possessing a system of generators, for which the ascending chain condition holds for every subobject of its generators.

\subsection{Preliminaries on topoi}

Recall that a \emph{Grothendieck topos}, or simply a topos, is a category $\ta$, equivalent to the category of set valued sheaves over a site \cite{mm}.  The category ${\sf Sh}(X)$ of all set valued sheaves over a topological space $X$ is an example of a topos.

A famous theorem of Giraud \cite{mm} asserts that a category $\ta$ is a topos if and only if the following axioms G1-G5 hold:
\begin{itemize}
\item[(G1)] The category $\ta$ has finite limits.
\item[(G2)] The category $\ta$ has colimits and they commute with pullbacks. 
\item[(G3)] Let $A_i$, $i\in I$ be a family of objects in $\ta$. The following
$$\xymatrix{ \emptyset\ar[r]\ar[d] & A_i\ar[d]\\A_j\ar[r] & \coprod A_i }$$
is a pullback diagram for any $i,j\in I$, such that $i\not=j$. Throughout this paper $\emptyset$ denotes the initial object of $\ta$.
\item[(G4)] Let $f:B\rightarrow A$ be an epimorphism and
$$\xymatrix{ C\ar[r]^{p_1}\ar[d]_{p_2} & B\ar[d]^f\\B\ar[r]_f & A }$$
a pullback diagram. The following
$$\xymatrix{ C \ar@<-.5ex>[r] \ar@<.5ex>[r] & B\ar[r]^f & A }$$
is a coequaliser.
\item[(G5)] There exist a set $I$ and a family of objects $G_i$, $i\in I$, called \emph{generators}, such that for any two parallel arrows $u,v:A\rightarrow B$ with $u\not=v$, there exist an $i\in I$ and an arrow $a:G_i\rightarrow A$, such that $ua\not=va$.
\end{itemize}

\begin{Le}\label{321} Let $A$ be an object of a topos $\ta$ and $B\subseteq A$. If any morphism $f:G_i\rightarrow A$ factors through $B$, then $B=A$. 
\end{Le}

\begin{proof} Let $\alpha$ be the inclusion $B\rightarrow A$. We need to show that $\alpha$ is an isomorphism. Let us recall that a morphism in a topos is an isomorphism, if and only if it is both a monomorphism and an epimorphism. To see that $\alpha$ is an epimorphism, let
$$\xymatrix{ B\ar[r]^\alpha\ar[d]_\alpha&A\ar[d]^{\beta_1}\\A\ar[r]_{\beta_2}& C }$$
be a pushout diagram. It suffices to show that $\beta_1=\beta_2$. Take any morphism $\gamma:G_i\rightarrow A$. By assumption, $\gamma=\alpha \delta$ for some morphism $\delta:G_i\rightarrow B$. Hence,
$$\beta_1\gamma=\beta_1\alpha\delta=\beta_2\alpha\delta=\beta_2\gamma$$
and we obtain $\beta_1=\beta_2$ since the objects $G_i$, $i\in I$ are generators. This finishes the proof.
\end{proof}

A \emph{geometric morphism}, or simply morphism, between topoi $f:\ta\rightarrow \ta_1$ consists of a pair of functors $f^\bullet:\ta_1\rightarrow \ta$ and $f_\bullet:\ta\rightarrow\ta_1$, called the \emph{inverse} and  \emph{direct image} respectively. These functors must satisfy the following two properties:
\begin{itemize}
\item[(i)] $f^\bullet$ is the left adjoint of $f_\bullet$, (in particular $f^\bullet$ commutes with all colimits, while $ f_\bullet$ commutes with all limits).
\item[(ii)] $f^\bullet$ commutes with finite limits.
\end{itemize}
The adjoint $f_\bullet$ (resp. $f^\bullet$) of a given functor $f^\bullet$(resp. $f_\bullet$), if it exists, is determined uniquely up to an isomorphism. As such, we can also define a geometric morphism $f:\ta\rightarrow \ta_1$ to simply be a functor $f^\bullet:\ta_1\rightarrow \ta$ that commutes with all colimits and with finite limits, and that has a left adjoint. There are many technical advantages in looking at geometric morphisms in this way, and we shall do so many times throughout this paper.

We mention as an example that any continuous map $f:X\rightarrow Y$ of topological spaces yields a pair of adjoint functors $(f^\bullet,f_\bullet):{\sf Sh}(X)\rightarrow {\sf Sh}(Y)$, defining a geometric morphism of topoi ${\sf Sh}(X)\rightarrow {\sf Sh}(Y)$, which is still denoted by $f$. Let $\mathcal F$ be a sheaf over $X$ and $\mathcal G$ a sheaf over $Y$. Recall that $f_\bullet({\mathcal F})(V)={\mathcal F}(f^{-1}(V))$ and $f^\bullet({\mathcal G})_x={\mathcal G}_{f(x)}$ for any open subset $V\subseteq Y$ and any $x\in X$.

Topoi and morphisms of topoi form a 2-category $\Top$. The 2-morphisms (or 2-cells) between morphism are natural transformations between the inverse image functors. The category of sets $\set$ is the 2-terminal object in this 2-category.

A \emph{point} of a topos $\ta$ is a geometric morphism $\set\rightarrow \ta$. The points of a topos $\ta$ form a category, which we denote by $\Pts(\ta)$. The isomorphism classes of $\Pts(\ta)$ will be denoted by ${\sf F}_\ta$.
Any morphism of topoi $f=(f^\bullet,f_\bullet):\ta\rightarrow\ta_1$ yields the functor $\Pts(f):\Pts(\ta)\rightarrow\Pts(\ta_1)$, which in turn induces the map ${\sf F}_f :{\sf F}_{\ta}\rightarrow {\sf F}_{\ta'}$.

\subsection{Localisations}

Let $\bf C$ and $\bf D$ be categories. A functor $j^\bullet:{\bf D}\rightarrow {\bf C}$ is called a \emph{localisation of categories} if it commutes with all colimits and finite limits. Further, $j^\bullet$ needs to posses a right adjoint $j_\bullet: {\bf C}\rightarrow \bf D$, which is full and faithful. It is well-known that, for the adjoint pair $j=(j^\bullet,j_\bullet):{\bf C}\rightarrow {\bf D}$, the functor $j^\bullet$ is a localisation of categories if and only if $j^\bullet$ respects finite colimits and the counit of the adjunction  $j^\bullet\circ j_\bullet\rightarrow \id_{\bf C}$ is an isomorphism. Denote by $\Sigma_j$ the collection of morphisms $\alpha$ of $\bf D$ for which $j^\bullet(\alpha)$ is an isomorphism. It is also well-known that the induced functor ${\bf D}[\Sigma_j^{-1}]\rightarrow \bf C$ is an equivalence of categories. Here, ${\bf D}[\Sigma_j^{-1}]$ is the category (if it exists) which is obtained by inverting the arrows in $\Sigma_j$.

Assume $j=(j^\bullet,j_\bullet):{\bf C}\rightarrow {\bf D}$  is an adjoint pair of functors, such that $j^\bullet$ is a  localisation of categories.  It is well-known that $\bf C$ is a topos if $\bf D$ is a topos. Hence, $j=(j_\bullet,j^\bullet)$ defines a morphism of topoi ${\bf C} \rightarrow \bf D$ with a full and faithful direct image $j_\bullet$. Such geometric morphisms are called \emph{embeddings of topoi} in \cite{mm}. We shall, however, stick with the notation of \cite{ab} and refer to them as localisations of topoi. Recall that if $U$ be an open subset of a topological space $X$, the inclusion $j:U\rightarrow X$ induces a localisation of categories $j^\bullet:{\sf Sh}(X)\rightarrow {\sf Sh}(U)$.

We also have the following easy, but important fact.

\begin{Le}\label{points_after_loc} Let $j=(j^\bullet,j_\bullet): \ta\rightarrow \ta_1$ be a geometric morphism of topoi, for which $j^\bullet$ is a localisation of categories. The induced map ${\sf F}_j:{\sf F}_{\ta}\rightarrow {\sf F}_{\ta'}$ is injective.
\end{Le}

\begin{proof} Take two points $p=(p^\bullet,p_\bullet)$ and $q=(q^\bullet,q_\bullet)$ of the topos $\ta$. Assume there exist an isomorphism of functors $\xi:p^\bullet j^\bullet\rightarrow q^\bullet j^\bullet$. It induces an isomorphism $p^\bullet j^\bullet j_\bullet\simeq q^\bullet j^\bullet j_\bullet$. The result follows since $j^\bullet j_\bullet\simeq \id$.
\end{proof}

\subsection{Gluing}

Let $\ta_i$, $i=0,1,2$ be topoi and $j^\bullet_i:\ta_i\rightarrow \ta_0$, $i=1,2$ localisations of categories. The \emph{gluing} (compare with \cite[p. 439]{ab}) of $\ta_1$ and ${\ta}_2$, along ${\ta}_0$, is the category $\ta$, defined as follows: Its objects are triples $(A_1,A_2,\alpha)$, where $A_i$ is an object of ${\ta}_i$, $i=1,2$ and $\alpha:j^\bullet_1(A_1)\rightarrow j^\bullet_2(A_2)$ is an isomorphism in ${\ta}_0$. A morphism $(A_1,A_2,\alpha)\rightarrow (B_1,B_2,\beta)$ in ${\ta}$ is a pair $(\xi_1,\xi_2)$, where $\xi_i:A_i\rightarrow B_i$, is a morphism in ${\ta}_i$, $i=1,2$, such that $j^\bullet(\xi_2)\alpha=\beta j^\bullet(\xi_1)$. 
We have the following diagram
$$\xymatrix{ {\ta}\ar[r]^{\pi_1 ^\bullet}\ar[d]_{\pi_2^\bullet} & {\ta}_1\ar[d]^{j^\bullet_1}\\
			 {\ta}_2\ar[r]_{j^\bullet_2} & {\ta}_0, }$$
where the projection $\pi_i^\bullet:{\ta}\rightarrow {\ta}_i$ is given by $\pi_i^\bullet(A_1,A_2,\alpha)=A_i$, $i=1,2$. 

\subsection{Locally $s$-Noetherian topoi} \label{s-noetherian}

We aim to introduce a topos theoretical analogue of the classical locally noetherian abelian categories, see \cite{ab}, called locally $s$-noetherian topoi. We will prove that gluing of such topoi is again a locally $s$-noetherian topos.  Most of the material discussed in this section is a topos theoretical analogue of the technique developed by Gabriel in \cite[Chapter 4]{ab}. 

Recall that an object in abelian category is noetherian if it satisfies the ascending chain condition for its subobjects. A basic but fundamental property of abelian categories is the one to one correspondence between subobjects and quotients. Hence, noetheriannes can be equivalently defined in terms of quotient objects.

This is no longer true for topoi as there are congruences that are not induced by subobjects. As such, there are two non-equivalent analogues of a noetherian object in topoi. One is based on subobjects and the other on quotients. We shall only study the first concept in this paper. We will refer to them as $s$-noetherian objects to make this distinction clear. 

To summarise, an object $A$ of a topos ${\ta}$ is called \emph{$s$-noetherian} if for any ascending sequence of subobjects
$$A_1\subseteq A_2\subseteq \cdots \subseteq A$$
of $A$, there is an integer $n$, such that $A_k= A_{k+1}$ for all $k\geq n$. 

\begin{Le} \label{snoto}
\begin{itemize}
\item[i)]  An object $A$ of a topos $\ta$ is $s$-noetherian if and only if any family of subobjects $(B_\lambda\subseteq A)_{\lambda\in \Lambda}$ has a maximal element.
\item[ii)] Let $f:A\rightarrow B$ be a monomorphism in $\ta$. If $B$ is $s$-noetherian, then $A$ is $s$-noetherian.
\item[ iii)]  Let $f:A\rightarrow B$ be an epimorphism in $\ta$. If $A$ is $s$-noetherian, then $B$ is $s$-noetherian.
\item[iv)] If $A$ and $B$ are $s$-noetherian, then $A\coprod B$ is $s$-noetherian.
\item [v)] If $A$ and $B$ are $s$-noetherian subobjects of $C$, then $A\cup B$ is an $s$-noetherian subobject of $C$.
\end{itemize}
\end{Le}

\begin{proof} i) Assume $A$ is $s$-noetherian and the family $(B_\lambda)$ does not contain a maximal element. Choose $\lambda_1\in \Lambda$ and consider
the corresponding subobject $B_{\lambda_1}$. It is not maximal by our assumption.  Thus, there exists $\lambda_2\in \Lambda$ with $B_{\lambda_1}\subsetneq B_{\lambda_2}$. We obtain a non-stabilising, ascending sequence of subobjects of $A$. This contradicts our assumption.

Conversely, assume the maximality condition holds and let $A\in\ta$ be an object. Consider any ascending sequence of subobjects
$$A_1\subseteq A_2\subseteq \cdots \subseteq A.$$
This family has a maximal element by assumption. Call it $A_n$. We have $A_k\subseteq A_n$ for all $k$. Thus $A$ is $s$-noetherian since $A_n=A_{n+1}=\cdots$.

ii) This is trivial since any subobject of $A$ is also a subobject of $B$.

iii) Take any ascending sequence of subobjects 
$$B_1\subseteq B_2\subseteq \cdots \subseteq B.$$
By assumption, there exists a natural number $n$, such that $f^{-1}(B_n)=f^{-1}(B_{n+1})=\cdots$. We have $ff^{-1}(B_k)=B_k$ since $f$ is an epimorphism. The result follows.

iv) Let
$$C_1\subseteq C_2\subseteq \cdots \subseteq C=A\coprod B$$
be any ascending sequence of subobjects of $A\coprod B$. Set $A_i=C_i\cap A$ and $B_i=C_i\cap B$. Both $(A_i)$ and $(B_i)$ form ascending sequences of subobjects of $A$ and $B$ respectively. As such, they stabilise. It follows that  $A_n=A_{n+1}=\cdots$ and $B_n=B_{n+1}=\cdots$ for some $n$. Since $C_k=A_k\coprod B_k$, one obtains $C_n=C_{n+1}=\cdots$ and the result follows.

v) The last statement follows from iv) and iii).
\end{proof}

\begin{De} A topos ${\ta}$ is said to be \emph{locally $s$-noetherian}, provided it posses a family of $s$-noetherian generators.
\end{De}

\begin{Rem} The product of two $s$-noetherian objects  in a locally $s$-noetherian topos
is not $s$-noetherian in general, see Example \ref{prnoe} below.
\end{Rem}

\begin{Le}\label{nounion} Let $\ta$ be a locally $s$-noetherian topos. One has $A=\bigcup_NN$ for any object $A$ of $\ta$. Here, $N$ runs through all the $s$-noetherian subobjects of $A$.
\end{Le}

\begin{proof} We set $B=\bigcup_NN$, where $N$ runs through all the $s$-noetherian subobjects of $A$. It is clear that $B\subseteq A$. Take any noetherian generator $G_i$ and a morphism $\alpha:G_i\rightarrow A$. We know from Lemma \ref{snoto}, part iv) that $\im(\alpha)$ is an $s$-noetherian subobject of $A$. Hence, $\im(\alpha)\subseteq B$ and we obtain $B=A$ by Lemma \ref{321}.
\end{proof}

\begin{Le} Let $\ta$ be a locally $s$-noetherian topos and $(G_s)_{s\in S}$ a family of $s$-noetherian generators. Any $s$-noetherian object is a quotient of a finite coproduct of $G_s$, $s\in S$.
\end{Le}

\begin{proof} Let $N$ be an $s$-noetherian object in $\ta$. Consider the collection of all subobjects of $N$, which are quotients of finite coproducts of the $G_s$-s. This collection contains a maximal subobject $M\subseteq N$ thanks to Lemma \ref{snoto}, since $N$ is $s$-noetherian. Note that it does not need to be unique. Hence, there is a finite subset $R\subseteq S$ with an epimorphism $\alpha:\coprod_{t\in R} G_t\rightarrow M$. We claim that $M=N$.

If this is not the case, there exists a morphism $\beta:G_i\rightarrow N$ which does not factor trough $M$. This follows from Lemma \ref{321}. We obtain a map $f=(\alpha,\beta): (\coprod_{t\in R} G_t)\coprod G_i\rightarrow N$ such that $M\subseteq \im(f)$ is a proper inclusion. This contradicts the maximality of $M$, proving our result.
\end{proof}

\begin{Co} The isomorphism classes of $s$-noetherian objects in a locally $s$-noetherian topos form a set.
\end{Co}

\begin{Le} \label{233} Let $\ta_1$ be a locally $s$-noetherian topos and  $j^\bullet:\ta_1\rightarrow \ta_0$ a localisation of categories.
\begin{itemize}
\item[i)] The topos $\ta_0$ is locally $s$-noetherian.
\item[ii)] Let $B$ be any object of $\ta_1$ and $A\subseteq j^\bullet (B)$ any $s$-noetherian subobject. There is an $s$-noetherian subobject $N\subseteq B$, such that $j^\bullet (N)= A$.
\end{itemize}
\end{Le}

\begin{proof} i) Denote by $\xi:Id_{\ta_1}\rightarrow j_\bullet j^\bullet$ the unit of the adjoint pair $(j^\bullet,j_\bullet)$. Let $X$ be an object of $\ta_1$. It follows from the general properties of adjoint functors that we have a morphism $\xi_X:X\rightarrow j_\bullet j^\bullet X$ in $\ta_1$	satisfying the following: For any object $Y$ of $\ta_0$ and any morphism $\alpha: X\rightarrow j_\bullet Y$, there is a unique morphism $\eta: j^\bullet X\rightarrow Y$ in $\ta_0$, such that $\alpha= j_\bullet(\eta)\xi_X$. We refer to this property as the universality of $\xi$.

The counit $j^*j_*\rightarrow Id_{\ta_0}$ is an isomorphism since $j_*$ is full and faithful. This allows us to identify $j^\bullet j_\bullet Y$ with $Y$ for any object $Y$ of $\ta_0$.

Take an object $N\in \ta_1$, any subobject $Y\subseteq j^\bullet(N)$ and consider the following pullback diagram in $\ta_1$:
$$\xymatrix{ \xi_N^{-1} (Y)= X\ar[r]^i\ar[d] & N\ar[d]^{\xi_N}\\ j_\bullet (Y)\ar[r] & j_\bullet j^\bullet (N). }$$
Our first claim is that one has $j^\bullet (X)=Y$. We show this by applying $j^\bullet $ to the above diagram to obtain the following commutative diagram:
$$\xymatrix{ j^\bullet (X)\ar[r]^{j^\bullet (i)}\ar[d] & j^\bullet (N)\ar[d]^\id\\ j^\bullet (j_\bullet (Y))\ar[r] & j^\bullet (j_\bullet j^\bullet (N))=j^\bullet(N). }$$
Since $j^\bullet $ respects all finite limits, this is in fact a pull-back diagram. We observe that the right vertical arrow is an isomorphism. It follows that the left vertical arrow $j^\bullet (X)\rightarrow j^\bullet j_\bullet Y=Y$ is an isomorphism as well since this is a pullback diagram. This proves the first claim.

Let $(G_s)_{s\in S}$ be a family of generators of $\ta_1$. Our second claim is that $(j^\bullet (G_s))_{s\in S}$ are generators of $\ta_0$. Take two morphisms $\alpha,\beta:A\rightarrow B$ in $\ta_0$ and assume $\alpha\not=\beta$. Since $j_\bullet $ is full and faithful, $j_\bullet (\alpha)\not=j_\bullet (\beta)$. The collection $\{G_s\}_{s\in S}$ is a family of generators, which means that there exists a morphism $\gamma:G_s\rightarrow j_\bullet A$ such that $j_\bullet (\alpha)\gamma\not=j_\bullet (\beta)\gamma$. We have $\gamma=j_\bullet (\delta)\xi_{G_s}$ for some $\delta:j^\bullet (G_s)\rightarrow A$, by the universality property of $\xi$. Thus, $j_\bullet (\alpha\delta)\xi_{G_s}\not=j_\bullet (\beta\delta)\xi_{G_s}$ and hence, $\alpha\delta\not=\beta\delta$. This proves the second claim.

Our third claim is that the image $j^\bullet (N)$ of an $s$-noetherian object $N\in\ta_1$ is an $s$-noetherian object in $\ta_0$. We consider the sequence of subobjects
$$A_1\subseteq A_2\subseteq \cdots \subseteq j^\bullet (N)$$
in $\ta_0$. Consider the corresponding sequence of subobjects of $N$ in $\ta_1$
$$\xi^{-1}_N(j_\bullet (A_1))\subseteq \xi_N^{-1}(j_\bullet (A_2))\subseteq \cdots\subseteq N.$$
Since $N$ is $s$-noetherian, there exist $n\geq 1$ such that $\xi^{-1}_N(j_\bullet (A_i))=\xi^{-1}_N(j_\bullet (A_{i+1}))$, for all $i\geq n$.  It follows from the first claim that $A_i=A_{i+1}$ as well. Thus, $j^\bullet (N)$ is $s$-noetherian. This, in conjunction with the second claim, implies the result.

ii) Take any object $B$ of $\ta_1$ and any $s$-noetherian subobject $A\subseteq j^\bullet (B)$. Consider subobjects $Y\subseteq A$ for which there exists an $s$-noetherian subobject $X\subseteq \xi_B^{-1}(j_\bullet A)$, such that $j^\bullet (X)=Y$. This family has a maximal object $Y'$ since $A$ is $s$-noetherian by i) Lemma \ref{snoto}. So, $Y'=j^\bullet (X')$ for an $s$-noetherian subobject $X'\subseteq \xi_B^{-1}(j_\bullet A)$. We have to prove that $Y'=A$.

Since $j_\bullet$ is full and faithful, it suffice to show that $j_\bullet Y'=j_\bullet A$. Assume it is not. Then $\xi^{-1}_Bj_\bullet (Y')\not=\xi^{-1}_B(j_\bullet (A))$ and by Lemma \ref{nounion}, there exists an $s$-noetherian subobject $Z\subseteq \xi^{-1}(j_*(A))$, such that $Z$ does not lie in $\xi^{-1}(j_\bullet (Y'))$. We set $W=Z\cup X'\subseteq \xi^{-1}(j_\bullet (A))$, which is $s$-noetherian by v) Lemma \ref{snoto}. It follows that $j^\bullet (W)$ is also a member of the above collection, which contradicts the maximality of $Y'$. This shows that $Y'=A$ and the lemma follows.
\end{proof}

\subsection{Gluing of locally $s$-noetherian topoi} \label{gluing s-noetherian}

\begin{Th}\label{Int} Assume $\ta_1$ and $\ta_2$ are locally $s$-noetherian topoi and the category $\ta$ is the gluing of $\ta_1$ and $\ta_2$ along the topos $\ta_0$:
$$\xymatrix{ {\ta}\ar[r]^{\pi^\bullet_1}\ar[d]_{\pi^\bullet_2} & {\ta}_1\ar[d]^{j^\bullet_1}\\
			 {\ta}_2\ar[r]_{j^\bullet_2} & {\ta}_0. }$$
Here, $j_1^\bullet$ and $j_2^\bullet$ are localisations of categories. Then $\ta$ and $\ta_0$ are locally $s$-noetherian topoi and the functors $\pi_i^\bullet$, $i=1,2$ are localisations.
\end{Th}

\begin{proof} We showed in part i) of Lemma \ref{233} that $\ta_0$ is a locally noetherian topos. One easily sees that $\ta$ posses all colimits and finite limits and that the functors $\pi_i^\bullet$, $i=1,2$ preserve colimits and finite limits. Based on this, one easily checks that the Giraud Axioms G1-G4 hold. To check Axiom G5 we proceed as follows:

The functors $j^\bullet_i:\ta_i\rightarrow \ta_0$, $i=1,2$ have, by assumption, right adjoint functors $j_{i\bullet}:\ta_0\rightarrow \ta_i$, $i=1,2$. We may assume, without loss of generality, that $j^\bullet_ij_{\bullet i}(A_0)=A_0$ for any object $A_0$ of $\ta_0$. Let $\xi_i:Id_{\ta_i}\rightarrow j_{\bullet i}j_i^\bullet$ be the unit of the adjoint pair $(j_i^\bullet,j_{\bullet i})$, $i=1,2$. Define the functors
$$\pi_{\bullet i}:\ta_i\rightarrow \ta, \  \ i=1,2$$
by
$$\pi_{\bullet 1}(A_1)=(A_1, j_{\bullet 2}(j^\bullet_1(A_1)), \id_{j^\bullet _1(A_1)})$$
and 
$$\pi_{\bullet 2}(A_2)=(j_{\bullet 1}(j^\bullet_2(A_2)),A_2, \id_{j^\bullet _2(A_2)}).$$
This makes $\pi_{\bullet i}$ the right adjoint of $\pi^\bullet_i$, for $i=1,2$. The latter are localisations since $\pi^\bullet_i\pi_{\bullet i}=Id_{\ta_i}$, $i=1,2$.

It is obvious that $(A_1,A_2,\alpha)$ is an $s$-noetherian object in $\ta$, provided both $A_1$ and $A_2$ are $s$-noetherian in $\ta_i$. We claim that the converse is also true. We will need the following construction for the proof:

Let $(A_1,A_2,\alpha)$ be an object of $\ta$ and let $X\subseteq A_1$ be a subobject. Then $(X, Y, \gamma)$ is a subobject of $(A_1,A_2,\alpha)$ in $\ta$. Here, $Y=\xi_2^{-1}(\alpha(j^\bullet_1(X)))$ and $\gamma$ is the restriction of $\alpha$ on $j^\bullet_1(X)$. It is clear from this construction that any ascending chain of subobjects of $A_1$ can be lifted as an ascending chain of subobjects of $(A_1,A_2,\alpha)$. Thus, $A_1$ is $s$-noetherian in $\ta_1$ if $(A_1,A_2,\alpha)$ is $s$-noetherian in $\ta$. Similarly for $A_2$. We have proven that an object $(A_1,A_2,\alpha)$ of the category $\ta$ is $s$-noetherian if and only if $A_1$ and $A_2$ are $s$-noetherian in $\ta_1$ and $\ta_2$ respectively. It follows that the isomorphism classes of $s$-noetherian objects of $\ta$ form a set. 

Hence, it suffices to show that $s$-noetherian objects generate $\ta$. Take any object $(A_1,A_2,\alpha)$  in $\ta$ and a strict subobject $(B_1,B_2,\beta)$ of $(A_1,A_2,\alpha)$. Either $B_1\not=A_1$ or $B_2\not= A_2$ and we may assume, without loss of generality, that we are in the first situation. In other words, that there exist an $s$-noetherian subobject $N_1\subseteq A_1$, which does not lie in $B_1$. There exist an $s$-noetherian subobject $N_2\subseteq A_2$ such that $j^\bullet_2(N_2)=\alpha(j^\bullet_1(N_1))$ by part ii) of Lemma \ref{233}. We have constructed an $s$-noetherian subobject $(N_1,N_2,\alpha_{|j^\bullet_1(N_1)})$ of $(A_1,A_2,\alpha)$ which does not factors through $(B_1,B_2,\beta)$. This shows that $s$-noetherian objects generate the category $\ta$, which finishes the proof.
\end{proof}

\section{Monoids and topoi}\label{020804}

This section deals with monoids and sets endowed with a monoid action. We will collect some well-known properties about these objects, which will be used in the forthcoming sections.

\subsection{Monoids and $M$-sets} Let $M$ be a multiplicatively written monoid.

We denote the subgroup of invertible elements of $M$ by $M^\times$. Recall that a subset $\m\subseteq M$ is called a \emph{right ideal} provided $\m M\subseteq \m$. A similar meaning has a \emph{left ideal}. An \emph{ideal} of $M$ is a subset $\m$ which is simultaneously a left and a right ideal.

A \emph{right $M$-set} is a set $A$, together with a map $A\times M\rightarrow A$, $(a,m)\mapsto am$, such that 
$$a1=1, \quad  {\rm and } \quad a(mn)=(am)n, \quad m,n\in M, a\in A.$$ 
A \emph{left $M$-set} has a similar meaning. It is clear that a right ideal of $M$ is a right $M$-set. We denote the category of left and right $M$-sets by $_M\set$ and $\set_M$ respectively. Observe that $_M\set\simeq \set_{M^{op}}$, where $M^{op}$ is the opposite monoid.

The category $\set_M$ (resp. $_M\set$) posses all limits and colimits. It is well-known that the forgetful functor from $\set_M$ (resp. $_M\set$) to $\set$ the category of sets preserves all limits and colimits. In particular, the coproduct is given by the disjoint union.

Recall also the construction of the tensor product of $M$-sets. Let $A$ be a right $M$-set and $B$ a left $M$-set. Denote by $A\otimes_MB$ the set  $(A\times B)/\sim$, where $\sim$ is the equivalence relation generated by $(am,b)\sim (a,mb)$. Here, $a\in A, b\in B$ and $m\in M$. The class of $(a,b)$ in $A\otimes_MB$ is denoted by $a\otimes b$. We have $am\otimes b=a\otimes mb$.

One has the following bijection
$$\Hom_\set(A\otimes_MB,C)\simeq \Hom_{\set_M}(A,\Hom_\set(B,C))$$
for any set $C$. We note that the $M$-set structure of $\Hom_\set(B,C)$ is given by
$$(f\cdot m)(b)=f(mb), \ \  m\in M, b\in B, f\in\Hom_\set(B,C).$$
It follows that the functor
$$(-)\otimes_MB:\set_M\rightarrow \set$$
has a right adjoint functor $ \Hom_\set(B,-)$ for any left $M$-set $B$. Hence, it 
commutes with arbitrary colimits. One can easily check that conversely, any covariant functor $F:\set_M\rightarrow \set$ commuting with all colimits is isomorphic to a functor of the type $(-)\otimes_MB$. Here, $B=F(M)$ as a set, while a left $M$-set structure on $B$ is induced by $mb:=F(l_m)(b)$, where $b \in F(M), m\in M$ and $l_m:B\rightarrow B$ is the homomorphism given by $l_m(b)=mb$.

It is well-known that both categories, $_M\set$ and $\set_M$, are topoi \cite{mm}.

\subsection{Monoid homomorphisms and morphisms of topoi}

Let $f:M\rightarrow N$ be a monoid homomorphism. We can consider a right $N$-set as a right $M$-set via this homomorphism. We denote it by $f^\bullet(A)$. Similarly for left $M$-sets. In particular, $N$ itself can be seen as both a right and a left $M$-set. Thus, we can form
$$f_\bullet(A)= \Hom_{\set_M}(N,A) \ \ {\rm and} \ \ f_!(A)=A\otimes _MN$$
for any right $M$-set $A$. Both of these sets have right $N$-set structures given by
$$ (g\cdot n)(x)=g(nx) \ \ \ \ {\rm and} \ \ \ \ (a\otimes n)x:=a\otimes (nx)$$
respectively. Here, $a\in A, n,x\in N$ and $g\in\Hom_{\set_M}(N,A)$. These $N$-sets define left and right adjoint functors of the functor $f^\bullet$. In particular, the pair $(f^\bullet,f_\bullet)$ defines a morphism of topoi $\set_M\rightarrow \set_N$, which will be denote by $\set_f$. We obtain a (pseudo)functor $$\set_f:{\sf Monoids}\rightarrow {\sf Top}.$$

\subsection{Localisation}\label{localisationsection}

We have defined two pairs of adjoint functors $(f^\bullet,f_\bullet)$ and $(f_!,f^\bullet)$ for a given homomorphism of monoids $f:M\rightarrow N$. In general, only the first one defines a morphism of topoi. We will show in this section that the the second adjoint pair is also a morphism of topoi, in the special case when $f$ is a localisation.

Let $M$ be a commutative monoid. Recall that if $S\subseteq M$ is a submonoid, one can form a new monoid $S^{-1}M$, called the \emph{localisation of $M$ by $S$}. Elements of $S^{-1}M$ are fractions $\frac{m}{s}$, where $m\in M$ and $s\in S$. By definition, $\frac{m_1}{s_1}=\frac{m_2}{s_2}$ if and only if there is an element $s\in S$, such that $m_1ss_2=m_2ss_1$.

This construction can be extended to $M$-sets in an obvious way. It is also obvious that $S^{-1}A$ is a natural $S^{-1}M$-set if $A$ is an $M$-set. We have a monoid homomorphism
$$f:M\rightarrow S^{-1}M,  \ \  f(m)=\frac{m}{1}.$$
One also has
$$S^{-1}A=A\otimes_M S^{-1}M=f_!(A)$$
for any $M$-set $A$. We discussed in the previous section that $f$ induces a geometric morphism of topoi
$$\set_f=(f^{\bullet},f_\bullet):\set_{M}\rightarrow \set_{S^{-1}M}.$$
One easily checks that the localisation $f_!$ preserves finite limits. Hence, the pair $(f_!,f^\bullet)$ defines a geometric morphism 
$$\set^f:\set_{S^{-1}M}\rightarrow \set_M.$$
For any $S^{-1}M$-set $A$, the $M$-set $f^\bullet(A)$ is isomorphic to $A$, considered as an $M$-set via the monoid homomorphism $f:M\rightarrow S^{-1}M$. It follows that $S^{-1}A\simeq A$, which implies that $f_!f^\bullet \simeq \id$. Hence, $\set^f:\set_{S^{-1}M}\rightarrow \set_M$ is a localisation of topoi, i.e. $f_!:\set_M\to \set_{S^{-1}M}$ is a localisation of categories.

\subsection{$s$-Noetherian monoids}

We have defined the notion of locally $s$-noetherian topoi in Section \ref{s-noetherian}, which we studied further in Section \ref{gluing s-noetherian}. Our aim now is to show that, under some finiteness assumptions on $M$, the topos of $M$-sets is $s$-noetherian. These types of monoids are going to be referred to as $s$-noetherian monoids. We will then continue to study them in a bit more detail.

Let $X$ be a set and $M$ a monoid. Recall that one can define a right $M$-set structure on $X\times M$ by $(x,n)m=(x,nm)$, $n,m\in M$, $x\in X$. A right $M$-set $A$ is called \emph{free} if there exists a set $X$ and an isomorphism of right $M$-sets $A\simeq X\times M$. It is well-known that any right $M$-set is a quotient of a free right $M$-set. We can consider $M$ itself as a free right $M$-set by taking $X$ to be a one-element set. We have $\Hom_{\set_M}(M,A)\simeq A$ for any $M$-set $A$. It follows that $A$ generates the topos $\set_M$.

We call a right $M$-set $A$ \emph{$s$-noetherian} if it satisfies the ascending chain condition on subobjects. According to part i) of Lemma \ref{snoto}, this happens if and only if any family of right $M$-subsets of $A$ contains a maximal member. A monoid $M$ is called \emph{right $s$-noetherian} if it is $s$-noetherian considered as a right $M$-set. I.e., if $M$ satisfies the ascending chain condition on right ideals. This is equivalent to saying that any family of right ideals of $M$ contains a maximal one.

\begin{Rem} We use the term $s$-noetherian in order to distinguish it from \emph{noetherian monoids} in the sense of \cite{gilmer}, which satisfy the ascending chain condition on congruences.
\end{Rem}

\begin{Le}\label{3651} Let $M$ be a monoid. The following are equivalent:
\begin{itemize}
\item[i)] $M$ is right $s$-noetherian.
\item[ii)] The topos $\set_M$ is $s$-noetherian.
\end{itemize} 
\end{Le}

\begin{proof} It is clear that $i)\Longrightarrow ii)$ since $M$ generates the topos $\set_M$.

Conversely, assume $\set_M$ is $s$-noetherian and let $(G_i)_{i\in I}$ be a set of noetherian generators of $\set_M$. Any object of $\set_M$ is a quotient of a coproduct of objects from the family of generators. In particular, there is an epimorphism $f:\coprod X_j\rightarrow M$, where each $X_j$ is isomorphic to one of the $G_i$-s. Since epimorphisms in $\set_M$ are exactly surjective homomorphisms of right $M$-sets, there is an element $x\in \coprod X_j$, such that $f(x)=1\in M$.

Recall that the coproduct in $\set_M$ is the disjoint union. Hence, $x\in X_j$ for some $j$. It follows that the restriction $f_j:X_j\rightarrow M$ of $f$ on $X_j$ is an epimorphism. Since $X_j$ is $s$-noetherian (it is isomorphic to one of the $s$-noetherian generator), it follows by part iii) of Lemma \ref{snoto} that $M$ is also $s$-noetherian. Thus, $(ii)\Longrightarrow i)$.
\end{proof}

\begin{Co}\label{s-noetherlocalisation} Let $M$ be an $s$-noetherian monoid and $S\subseteq M$ a multiplicative subset. The localisation $S^{-1}M$ is $s$-noetherian as well.
\end{Co}

\begin{proof} This follows directly from Lemmas \ref{233} and \ref{3651}, and the fact that as mentioned, the geometric morphism $\set^f:\set_{S^{-1}M}\rightarrow \set_M$ is a localisation of topoi.
\end{proof}

Recall that a right $M$-set $A$ is called \emph{finitely generated} if there are elements $a_1,\cdots, a_k\in A$, such that any element $a\in A$ can be written as $a=a_im$, for some $i\in\{1,\cdots,k\}$ and $m\in M$. A right ideal of $M$ is called finitely generated if it is so as a right $M$-set.

\begin{Le} A monoid $M$ is $s$-noetherian if and only if any right ideal of $M$ is finitely generated.
\end{Le}

\begin{proof} Assume $M$ is $s$-noetherian and there exists a right ideal $\m\subseteq M$, which is not finitely generated. In particular, $\m$ is non-empty. Take an element $m_1\in \m$ and consider the right ideal $\m_1$ generated by $m_1$. Clearly, $\m_1 \subsetneq \m$ and so, there exists an element $m_2\in \m$, such that $m_2\nin \m_1$. Let $\m_2$ be the right ideal generated by $m_1$ and $m_2$. We see that $\m_1 \subsetneq \m_2$. We obtain an infinite sequence of right ideals $\m_1 \subsetneq \m_2 \subsetneq \cdots$. This contradicts our assumptions. Hence, any right ideal is finitely generated.

Conversely, assume any right ideal is finitely generated. Let $\m_1 \subsetneq \m_2 \subsetneq \cdots$ be an ascending sequence of ideals and set $\m:=\bigcup \m_i$. By assumption, $\m$ is finitely generated and as such, there exists an $n\in\N$, such that $\m_n$ contains all the generators. Clearly, $\m_n=\m_{n+1}=\cdots$ and we are done.
\end{proof}

\begin{Exm}\label{snoe} It is clear that any finite monoid is $s$-noetherian. Moreover, any group is $s$-noetherian, because any non-empty right ideal of a group $G$ is $G$. 
The additive monoid of natural numbers $\N$ is also $s$-noetherian, because any ideal of $\N$ is principal. We refer to Lemma \ref{090217} for other examples.
\end{Exm}

\begin{Le} Let $M$ be an $s$-noetherian monoid and let $A$ be a right $M$-set.
\begin{itemize}
\item[i)] If $A$ is a finitely generated $M$-set and $B\subseteq A$ is an $M$-subset, then $B$ is also finitely generated.
\item[ii)] The $M$-set $A$ is $s$-noetherian if and only if $A$ is finitely generated.
\item[iii)] Let $A$ be a finitely generated $M$-set. Any family of $M$-subsets of $A$ contains a maximal member.
\end{itemize}
\end{Le}

\begin{proof} i)  The statement is true if $A\cong M$ is free with one generator, because $M$-subsets of $M$ are simply ideals. The union of finitely  generated $M$-subsets is finitely generated and the coproduct in the category $\set_M$ is given by the disjoint union. It follows that the statement is also true when $A$ is finitely generated and free as an $M$-set. In fact, we can think of $A$ as $A=\coprod_{i=1}^k A_i$, where each $A_i$ is free with one generator. Denote $B_i:=B\cap A_i$. These are $M$-subsets of the $A_i$-s and therefore, finitely generated as $M$-sets. We have $B=\coprod_{i=1}^k B_i$ and thus, $B$ is also a finitely generated $M$-set.

Let $A$ and $F$ be finitely generated $M$-sets and $\xi:F\rightarrow A$ a surjective homomorphism. Consider $\xi^{-1}(B)$. It is an $M$-subset of $F$ and as shown, finitely generated. It follows that $B$ is also finitely generated since $\xi^{-1}(B) \to B$ is surjective.

ii) Assume $A$ is finitely generated and let
$$A_1\subseteq A_2\subseteq\cdots\cdots\subseteq A$$
be a sequence of $M$-subsets of $A$. Consider the $M$-subset $B=\bigcup_{n=1}^\infty A_n$. It is finitely generated by part i). As such, there exists an integer $n$, such that $A_n$ contains every generator of $B$. We have $B=A_n$ and $A_n=A_{n+1}=\cdots$. Hence, $A$ is $s$-noetherian.

Conversely, assume $A$ is $s$-noetherian. This statement holds trivially if $A=\emptyset$. If $A\not=\emptyset$, we can take $a_1\in A$. Let $A_1$ be the $M$-subset generated by $a_1$. We are done if $A_1=A$. Otherwise, there exists an element $a_2\in A$, such that $a_2\nin A_1$. Denote by $A_2$ the $M$-subset generated by $a_1$ and $a_2$. We are done if $A_2=A$. Otherwise we continue. This process will stop after a finite number of steps by our assumption. Hence, $A$ is finitely generated.

iii) Let $A_i$, $i\in I$ be a family of $M$-subsets. Choose $i_1\in I$ and consider the corresponding $M$-subset $A_{i_1}$. We are done if $A_1$ is maximal. Otherwise, there exists $i_2\in I$ such that $A_{i_1}\subsetneq A_{i_2}$. We obtain an ascending sequence of $M$-subsets in this way, which must stabilise by part ii). Hence, $A_n=A_{n+1}=\cdots$ and as such, $A_n$ is maximal.
\end{proof}

\begin{Le} \label{090217} Let $M$ be a monoid.
\begin{itemize}
\item[i)] Let $f:M\rightarrow N$ be a surjective homomorphism of monoids. If $M$ is $s$-noetherian, then $N$ is $s$-noetherian.
\item[ii)] If $M$ is $s$-noetherian, then $M\times \N$ is also $s$-noetherian.
\item[iii)] $M$ is $s$-noetherian if it is commutative and finitely generated.
\item[iv)] Let $M$ be a commutative monoid. Then $M$ is $s$-noetherian if and only if $M_\red= M/M^\times$ is $s$-noetherian.
\end{itemize}
\end{Le}

\begin{proof} i) Take a right ideal $\n$ of $N$. The preimage $f^{-1}(\n)$ is a right ideal of $M$ and hence, finitely generated. The images of these generators generate $\n$.

ii) Consider the canonical projection $\pi:M\times \N\rightarrow M$, given by $\pi(m,t)=m$, $m\in M$ and $t\in \N$. Let $\m$ be a right ideal of $M\times \N$. Its image $\pi(\m)$ is a right ideal of $M$ since $\pi$ is surjective. By assumption, there are elements
$$(m_1,t_1),\cdots, (m_k,t_k)\in \m,$$
such that $m_1,\cdots,m_k$ generate $\pi(\m)$. Let $t={\sf max}\{t_1,\cdots,t_k\}\in \N$. For a natural number $s\leq t$, we consider the subset $\m_s$ of $\m$ given by
$$\m_{s}=\{m\in M \ | \ (m,s)\in \m\}.$$
It is not only a subset of $M$, but also a subideal of $\pi(\m)$, because if $(m,s)\in \m$ and $m'\in M$, then $(mm',s)=(m,s)(m',0)\in \m$ (recall that we use additive notations for $\N$). Denote by $m_{1s},\cdots ,m_{l_s}$ the generators of this ideal. We claim that the elements
$$(m_i,t_i), (m_{js},s), \ i=1,\cdots,k, 0\leq s\leq t, 1\leq j\leq l_s$$
generate $\m$. To prove this claim, let $\n\subseteq \m$ denote the ideal generated by these elements. We have to show that $\m\subseteq \n$. Take an element $(m,r)\in \m\subseteq M\times \N$. If $r\leq t$, then $m\in \m_r$. So, we can write $m=m_{jr}m'$ for some $j$ and $m'\in M$. This gives us $(m,r)=(m_{jr},r) (m',0)$. We have shown that $(m,r)\in \n$. Assume now that $r>t$. Since $m\in \pi(\m)$, we can write $m=m_im'$, for some $i\in\{1,\cdots, k\}$ and $m'\in M$. We have
$$(m,r)=(m_i,t_i)(m',r-t_i)\in \n,$$ 
and we are done.

iii) By Example \ref{snoe}, $\N$ is $s$-noetherian. It follows from part ii) that the monoid $\N^d$ is $s$-noetherian for all $d\geq 0$. We can use part i) to finish the proof.

iv) By part i), we only have to show that $M$ is $s$-noetherian if $M_\red$ is $s$-noetherian. Consider the canonical projection $\pi:M\rightarrow M_\red$. The image $\pi(\m)$ of any ideal $\m\subseteq M$ is a finitely generated ideal of $M_\red$. There exists a finite collection of elements $m_1,\cdots,m_k\in \m$, such that $\pi(m_1),\cdots,\pi(m_k)$ generate $\pi(\m)$. Take any element $a\in \m$. There is an $i\in\{1,\cdots,k\}$ and an $m\in M$, such that $\pi(a)=\pi(m_i)\pi(m)$. Hence, $ag=m_im$ for some $g\in M^\times$. We have $a=m_i(mg^{-1})$ and hence, $m_1, \cdots,m_k$ generate $\m$.
\end{proof}

To see why we are working with locally $s$-noetherian topoi, rather than $s$-noetherian, observe the following example:

\begin{Exm}\label{prnoe} Let $M=\mathbb{N}$ be the monoid of natural numbers. The category $\set_\N$ is a locally $s$-noetherian topos by Lemma \ref{3651}. Clearly, $A=\N$ is a free $\N$-set with one generator. Thus, it is $s$-noetherian. However, $B=A\times A=\N^2$ is not finitely generated as an $\N$-set and as such, not $s$-noetherian.
\end{Exm}

\section{Monoid schemes and topoi}\label{030804}

The aim of this section is to study the category of quasi-coherent sheaves in more detail. In particular, we will show that this category is a locally $s$-noetherian topos under some finiteness conditions on $X$. We will first collect some basic facts and definitions about monoid schemes. Our main references are \cite{Lorscheid},\cite{chww}, \cite{deitmar}  and \cite{p2}.

All monoids are assumed to be commutative and written multiplicatively.

\subsection{Monoid schemes}\label{msch}

Let $M$ be a monoid. Recall that an ideal $\p$ is called \emph{prime} if for any $a,b\in M$ with $ab\in \p$, one has $a\in \p$ or $b\in \p$. We denote the set of all prime ideals of $M$ by $\Spec(M)$ \cite{deitmar}, \cite{p1}. One defines
$$D(f)=\{\p\in \Spec(M) \ | \ f\nin \p\}$$
for any element $f\in M$. The family of sets $D(f)$ forms a basis of open sets of the standard topology on $\Spec(M)$. Let $A$ be an $M$-set. It is well-known that there exist a unique sheaf of sets $\tilde{A}$ over $\Spec(M)$, such that
$$\Gamma(D(f),\tilde{A})=A_f.$$
The section of a sheaf $\mathcal{F}$ on an open subset $U$ (i.e. $\mathcal{F}(U)$) will also be denoted by $\Gamma(U, \mathcal{F})$. The localisation of $A$ with respect to a submonoid of $M$ generated by $f$ will be denoted by $A_f$. The stalk of $\tilde{M}$ at the point $\p$ is the localisation $A_\p$ of $A$ by the submonoid $M\setminus \p$. We use the notation $\O_M$ for the sheaf $\tilde{M}$ in the special case when $A=M$. This is a sheaf of monoids, while the sheaf $\tilde{A}$ becomes a sheaf of $\O_M$-sets, for any $M$-set $A$.

A monoid homomorphism $f:M\rightarrow N$ is called \emph{local} provided $f^{-1}(N^\times)\subseteq M^\times$. As usual, $M^\times$ denotes the subgroup of invertible elements of $M$. Equivalently, $f$ is local if it sends non-invertible elements to non-invertible elements.

A \emph{monoid space} is a pair $(X, \O_X)$, where $X$ is a topological space and $\O_X$ is a sheaf of monoids on $X$. A \emph{morphism of monoid spaces} $f:(X, \O_X)\rightarrow (Y, \O_Y)$ is given by a continuous map $f:X\rightarrow Y$, together with a morphism of monoid sheaves $f^\#:\O_Y \rightarrow f_\bullet\O_X$. Additionally, the induced morphisms on stalks $f^\#_x:\O_{Y, f(x)}\rightarrow \O_{X,x}$ must be local for all $x\in X$. We will often simply write $X$ instead of $(X,\O_X)$. The category of monoid spaces will be denoted by ${\bf MSpaces}$.

A monoid scheme is \emph{affine} provided it is isomorphic to $(\Spec(M),\O_M)$. The assignment $M\mapsto (\Spec(M),\O_M)$ yields a contravariant embedding of the category $\bf Com.Mon$ of commutative monoids into the category ${\bf MSchemes}$ of monoid schemes. 

A \emph{monoid scheme} is a locally affine monoid space $(X, \O_X)$. In other words, every point has an open neighbourhood isomorphic to the monoid space $(\Spec(M),\O_M)$, for some monoid $M$ \cite{deitmar}, \cite{chww}. Monoid schemes form a full subcategory of ${\bf MSpaces}$, which will be denoted by $\bf MSchemes$.

A monoid scheme is called $s$-\emph{noetherian} if it can be covered by a finite number of open affine monoid subschemes $\Spec(M_i)$, where each $M_i$ is an $s$-noetherian monoid.

Let $(X,\O_X)$ be a monoid scheme and $U\subseteq X$ an open subset. The restriction of $\O_X$ to $U$ is called an \emph{open subscheme} and is often denoted by $(X,\O_{X|U})$. 

\begin{Le}\label{trevesi} An open subscheme of an $s$-noetherian monoid scheme $X$ is $s$-noetherian.
\end{Le}

\begin{proof} We first consider the case when $X=\Spec(M)$ is affine. Let $U$ be an open subset of $X$. There is an ideal $\m$ such that $U=X\setminus V(\m)$,  where $V(\m)=\{\p\in \Spec(M) \ | \ \m \subseteq \p\}$. Since $\m$ is finitely generated, there are $f_1,\cdots,f_k\in M$ such that $U=D(f_1)\cup\cdots \cup D(f_k)$. We have $D(f_i)=M_{f_i}$ for all $i\in\{1,\cdots,k\}$, where the $M_{f_i}$-s are $s$-noetherian by Corollary \ref{s-noetherlocalisation}. This makes $U$ an $s$-noetherian monoid scheme.

We have $X=U_1\cup\cdots\cup U_k$ in the general case, where the $U_i$-s are open and affine. Let $U$ be an open subscheme of $X$. We have $U=\cup_{i=1}^k(U\cap U_i)$. Each $U\cap U_i$ is an open monoid subscheme of an affine $s$-noetherian monoid scheme. As shown above, they can be expressed as a finite union of open and affine $s$-noetherian monoid subschemes. Hence, $U$ itself can be expressed in such a way, implying the result.
\end{proof}

\subsection{Sheaves of $\O_X$-sets}

Let $X$ be a monoid scheme and $\mathcal{A}$ a sheaf of $\O_X$-sets. That is to say, $\mathcal{A}$ is a sheaf of sets, together with an action of the monoid $\O_X(U)$ on the set $\mathcal{A}(U)$ for all open  $U\subseteq X$. Further, these actions are compatible for varying $U$. One denotes by $\O_X$-${\sf Sets}$ the category of sheaves of $\O_X$-sets.  

\begin{Le}\label{yonshe} Let $\mathcal{A}$ be a sheaf of $\O_X$-sets. One has the following isomorphism
$$ \Hom_{\O_X}(\O_X,\mathcal{A})\simeq \Gamma(X,\mathcal{A}).$$
\end{Le}

\begin{proof} Recall that a morphism $\alpha:\O_X\rightarrow \mathcal{A}$ of the category $\O_X$-${\sf Sets}$ is a collection of maps $\alpha_U:\O_X(U)\rightarrow \mathcal{A}(U)$, where $U$ is running through the open subsets of $X$. One easily sees that the assignment $\alpha\mapsto\alpha_X(1)$ induces the isomorphism in question.
\end{proof}

It is well-known that the category $\O_X$-${\sf Sets}$ is a topos. The \emph{tensor product} of two $\O_X$-sets $\mathcal{A}$ and $\mathcal{B}$ is the sheafification of the presheaf of $\O_X$-sets given by $U\mapsto{\mathcal{A}}(U)\otimes_{\O_X}{\mathcal{B}}(U)$. Here, $U\subseteq X$ is an open subset of $X$.

Let $(f,f^\#):X\rightarrow Y$ be a morphism of monoid schemes. As already mentioned, the continuous map $f$ induces a geometric morphism of topoi $(f^\bullet, f_\bullet):{\sf Sh}(X)\rightarrow{\sf Sh}(Y)$. The functor $f_\bullet$ respects inverse limits. As such, the sheaf $f_\bullet {\mathcal O}_X$ is a sheaf of monoids on $Y$ and for any $\O_X$-set ${\mathcal A}$, the sheaf $f_\bullet \mathcal{A}$ is a sheaf of $f_\bullet\O_X$-sets. The morphism of monoid sheaves $f^\#:\O_Y\rightarrow f_\bullet\O_X$ allows us to define an action of $\O_Y$ on $f_\bullet \mathcal{A}$. We obtain the functor
$$f_\bullet: \O_X\textnormal{-}{\sf Sets} \rightarrow \O_Y\textnormal{-}{\sf Sets}, \quad \textnormal{given by} \quad \mathcal{A}\mapsto f_\bullet \mathcal{A}.$$
The functor $f_\bullet$ has a left adjoint functor $f^*$, which is constructed as follows. Take an $\O_Y$-set $\mathcal{B}$. Since $f^\bullet:{\sf Sh}(Y)\rightarrow {\sf Sh}(X)$ respects finite limits, $f^\bullet \mathcal{B}$ has an $f^\bullet \O_X$-set structure.

On the other hand, the morphism of monoid sheaves $f^\#:\O_Y\rightarrow f_\bullet \O_X$ induces (by adjointnes) the morphism of monoid sheaves $f^\bullet \O_Y\rightarrow\O_X$. Hence, $\O_X$ can be considered as an $f^\bullet \O_Y$-set and one has
$$f^*(\mathcal{B}):=f^\bullet{\mathcal{B}}\otimes_{f^\bullet \O_Y}{\O_X}.$$

\begin{Le}\label{open-o} Let $U$ be an open subscheme of a monoid scheme $X$. The inclusion $j:U\rightarrow X$ induces a localisation of topoi $j=(j^*,j_*):\O_U\textnormal{-}{\sf Sets} \rightarrow \O_X$-${\sf Sets}$, i.e. $j^*:\O_X$-${\sf Sets} \to \O_U\textnormal{-}{\sf Sets}$ is a localisation of categories.
\end{Le}

\begin{proof} In this case, $j^*\O_X=\mathcal{U}_X$. Hence, $j^*=j^\bullet$ and the result follows from our discussion at the end of Section \ref{localisationsection}.
\end{proof}

We will usually use $f_*$ instead of $f_\bullet$, in order to synchronise the notation with $f^*$. Whenever used, we will always have $f_*:=f_\bullet$.

\subsection{Quasi-coherent sheaves}

Let $\mathcal{A}$ be a sheaf of $\O_X$-set. It is said to be \emph{quasi-coherent} \cite{Lorscheid} if there is an open and affine subscheme $U=\Spec(M)$ for any point $x\in X$, such that $x\in U$ and the restriction of $\mathcal{A}$ on $U$ is isomorphic to a sheaf of the form $\tilde{A}$. Assume $\mathcal{A}$ is quasi-coherent. It is well-known \cite[Theorem 3.12]{Lorscheid} that the restriction of $\mathcal{A}$ on any open affine subscheme $U=\Spec(M)$ is isomorphic to a sheaf of the form $\tilde{A}$, for some $M$-set $A$.

The category of quasi-coherent sheaves over $X$ is denoted by $\Qc(X)$. It posses all colimits and finite limits since the localisation functor respects them \cite[Lemma 3.16]{Lorscheid}. If $X=Spec(M)$ is affine, then $A\mapsto \tilde{A}$ defines an equivalence of categories $\Qc(X)\simeq M$-$sets$, see \cite[Corollary 3.13]{Lorscheid}.
   
\begin{Le}\label{qcf} Let $f:X\rightarrow Y$ be a morphism of monoid schemes. 
\begin{itemize}
\item[i)] If $\mathcal{B}$ is a quasi-coherent sheaf of $\O_Y$-sets, then $f^*\mathcal{B}$  is a quasi-coherent sheaf of $\O_X$-sets.
\item[ii)] Assume $X$ is $s$-noetherian. If $\mathcal{A}$ is a quasi-coherent sheaf of $\O_X$-sets, then $f_*\mathcal{A}$ is a quasi-coherent sheaf of $f_*\O_X$-sets.
\end{itemize}
\end{Le}

\begin{proof} i) Take $x\in X$ and consider $f(x)\in Y$. Since $\mathcal{B}$ is quasi-coherent, there is a on open and affine subscheme $V=\Spec(N)\subseteq Y$, such that $f(x)\in V$ and the restriction of $\mathcal{B}$ on $V$ is of the form $\tilde{B}$, for some $N$-set $B$. Choose an open and affine subscheme $U=\Spec(M)$ of $X$, such that $f(U)\subseteq V$. The restriction of $f$ gives a morphism $U\rightarrow V$. This is determined by a monoid homomorphism $\phi:N\rightarrow M$. The result follows since $f^*(\mathcal{B})=\widetilde{M\otimes_N B}$.

ii) First, consider the case when both $X=\Spec(M)$ and $Y=\Spec(N)$ are affine. In this case, $f$ is induced by $\phi:N\rightarrow M$. Moreover, $\mathcal{A}=\tilde{A}$ for an $M$-set $A$ and $f_*\mathcal{A}=\tilde{A}$. Here, $N$ acts on $A$ via $\phi$ and the result follows.

Let $V$ be an open subset of a general monoid scheme $Y$. Then $U=f^{-1}(V)$ is $s$-noetherian by Lemma \ref{trevesi}. Hence, we can cover it by a finite number of open affine subschemes $U_i$, $i=1,\cdots,n$. Moreover, we can cover each $U_i\cap U_j$ by a finite number of open affine subschemes $U_{ijk}$.
The restriction of $f_*\mathcal{A}$ on $V$ is the equaliser of the diagram 
$$\xymatrix{ \prod_i f_*(\mathcal{A}_{|U_i}) \ar@<-.5ex>[r] \ar@<.5ex>[r] & \prod_{ijk} f_*(\mathcal{A}_{|U_{ijk}}), }$$
as $\mathcal{A}$ is a sheaf. Since the statement is true in affine case, $f_*(\mathcal{A}_{|U_i})$ and $f_*(\mathcal{A}_{|U_{ijk}})$ are quasi-coherent sheaves. We also know that a finite limit of quasi-coherent sheaves is again quasi-coherent \cite[Lemma 3.16]{Lorscheid}. It follows that $f_*\mathcal{A}$ is quasi-coherent since the product, as well as the equaliser, are specific types of limits.
\end{proof}

\begin{Le} Let $X$ be an $s$-noetherian monoid scheme and $U$ an open subscheme. The inclusion $j:U\rightarrow X$ induces a localisation of categories $j^*:\Qc(X)\rightarrow \Qc(U)$.
\end{Le} 

\begin{proof} The restriction of $j=(j^*,j_*):\O_U\textnormal{-}{\sf Sets}\rightarrow\O_X\textnormal{-}{\sf Sets}$ on quasi-coherent sheaves yields a well-defined morphism of topoi by Lemma \ref{qcf}. The morphism $j^*:\O_U\textnormal{-}{\sf Sets}\rightarrow\O_X\textnormal{-}{\sf Sets}$ is a localisation by Lemma \ref{open-o}. Hence, the same is true for $j^*:\Qc(X)\rightarrow \Qc(U)$.
\end{proof}

The following theorem might be of independent interest, as it introduces a previously unknown class of topoi, which is very closely related to the topoi of $M$-sets.

\begin{Th}\label{432.05.03} Let $X$ be an $s$-noetherian monoid scheme. The category $\Qc(X)$ of quasi-coherent sheaves is a locally $s$-noetherian topos.
\end{Th}

\begin{proof} Let $X=\Spec(M)$ be an affine monoid scheme. In this case, $\Qc(X)$ is equivalent to the category of $M$-sets. The result follows from Lemma \ref{3651}.

For a general monoid scheme $X$, observe that if $U$ and $V$ are open monoid subschemes of $X$, the category $\Qc(U\cup V)$ is equivalent to the gluing of $\Qc(U)$ and $\Qc(V)$ along $\Qc(U\cap V)$. An obvious induction argument, together with Theorem \ref{Int}, implies the result.
\end{proof}

\begin{Rem} The subobject classifier \cite{mm} of the topos $\Qc(X)$, where $X$ is an $s$-noetherian monoids scheme, is obtained by gluing the subobject classifiers of the topoi of quasi-coherent sheaves over affine subschemes. A similar statement for the exponential objects is not true in general. We will not need these facts. Hence, we omit the arguments.
\end{Rem}

\subsection{Line bundles}\label{44123sec}

Recall that a quasi-coherent sheaf $\mathcal{L}$ over a monoid scheme $X$ is called a \emph{line bundle} if it is locally isomorphic to $\O_X$ \cite{p2}. Let $X$ be a monoid scheme and $x\in X$ a point of its underlying space. Denote the maximal ideal of $\O_{X,x}$ by $\m_x$. This is the set of non-invertible elements of the monoid $\O_{X,x}$. 

Let $\mathcal{L}$ be a line bundle over $X$ and $f\in \Gamma(X,\mathcal{L})$. We define
$$X_f:=\{x\in X \ | \ f_x\nin\m_x\mathcal{L}_x \},$$
where $f_x$ is the restriction of $f$ on $\O_{X,x}$. Observe that $X_f$ is an open subset of $X$. We have $X_f=D(f)$ when $X=\Spec(M)$ is affine and $\mathcal{L}=\O_X$.

We will make use of the obvious paring in the following lemma:
$$\Gamma(X,\mathcal{A})\otimes_{\Gamma(X,\O_X)}\Gamma(X,\mathcal{B})\rightarrow\Gamma(X,\mathcal{A}\otimes_{\O_X}\mathcal{B}).$$
The image of $a\otimes b$ under this map is denoted by $ab$.

\begin{Le}\label{44123} Let $X$ be an $s$-noetherian monoid scheme, $\mathcal{L}$ a line bundle on $X$ and $f\in \Gamma(X,\mathcal{L})$. Moreover, let $\mathcal{F}$ be a quasi-coherent sheaf over $X$.
\begin{itemize}
\item [i)] Let $s_1,s_2\in \Gamma(X,\mathcal{F})$ be sections such that restricted on $X_f$, they are equal. There exist a natural number $k>0$, such that $f^ks_1=f^ks_2$ as elements of $\Gamma(X,\mathcal{L}^{\otimes k}\otimes \mathcal{F})$.
\item [ii)] For any $t\in  \Gamma(X_f,\mathcal{F})$, there exists $n>0$ such that $f^nt\in  \Gamma(X_f,\mathcal{L}^{\otimes n}\otimes\mathcal{F})$ is in the image of the restriction map
$$\Gamma(X,\mathcal{L}^{\otimes n}\otimes\mathcal{F}) \rightarrow \Gamma(X_f,\mathcal{L}^{\otimes n}\otimes\mathcal{F}).$$
\end{itemize}
\end{Le}

\begin{proof} First consider the case when $X=\Spec(M)$ is affine. Any line bundle (or, more generally, any vector bundle) is trivial in this case, \cite{p2}. Hence, $\mathcal{L}\simeq \O_X$ and so, we can assume that $f\in M$ and $\mathcal{F}=\tilde{A}$ for an $M$-set $A$. Since $X_f=D(f)$, the restriction of $\mathcal{F}$ on $X_f$ is $\tilde{A}_f$, the quasi-coherent sheaf associated to the $M_f$-set $A_f$. The result now follows from the basic properties of localisation.

For the general case, we choose a finite open cover of $X$ by affine subschemes $X_1, \cdots X_m$.  In this case, $X_f=X_{1f}\cup \cdots \cup X_{mf}$.

To see part i), observe that for each $1\leq i \leq m$, there exists a $k_i$  such that the restrictions of $f^{k_i}s_1$ and $f^{k_i}s_2$ on $X_{if}$ coincide. Let $k={\sf max}_{1\leq i\leq m}k_i$ be the maximum of the $k_i$-s. We see that $f^ks_1$ and $f^ks_2$ agree with each other on every $X_{if}$. Hence, they agree on $X_f$. This proves the first assertion.

For the second part, choose a natural number $n_i$ for each $1\leq i \leq m$, such that $f^{n_i}t$ has an extension to $X_i$. We take $n'$ to be the maximal among $n_1,\cdots, n_m$. There exists an element $t_i\in\Gamma(X_i, \mathcal{L}^{\otimes n'}\otimes\mathcal{F})$ for each $1\leq i\leq m$, such that $t_i=f^{n'}t$ restricted to $X_{if}$. The first part of this lemma says that there exists an integer $n>0$, such that $f^nt, \cdots, f^nt_m$ glue to a single global element, which restricts to $f^nt$.
\end{proof}

\section{The points of the topos $\set_M$}\label{poaff}

The aim of this Section is to study the topos points of $\set_M$. We are able to relate such points with prime ideals in the commutative case in Proposition \ref{inj}. If $M/M^\times$ is additionally a finitely generated monoid, the topos points correspond exactly to the localisations $M_\p$ of $M$ with its prime ideals $\p\in\Spec(M)$. In particular, there is a one-to-one relation between the isomorphism class of the topos points of $\set_M$ and the actual points of the topological space $\Spec(M)$. 

\subsection{Filtered $M$-sets}

A left $M$-set $A$ is called \emph{filtered} provided the functor 
$$(-)\otimes_M A:_M\set\rightarrow \set$$ 
commutes with finite limits. Recall that it always commutes with all colimits. The collection of all filtered left $M$-sets is of course a full subcategory of the topos $_M\set$, which we denote by $_M\F$. According to Diaconescu's theorem \cite[Theorem VII. 2 on p. 381]{mm}, one has an equivalence of categories
$$\, _M\F \xrightarrow{\sim} \Pts(\set_M),$$
which sends a filtered left $M$-set $A$ to the (geometric) morphism $f_A:\set\rightarrow \set_M$. Here,
$$(f_A)^\bullet=(-)\otimes_MA:\set_M\rightarrow \set \quad {\rm and} \quad (f_A)_\bullet= \Hom_\set(A,-):\set\rightarrow \set_M.$$
The right $R$-set structure on $ \Hom_\set(A,B)$ is given by $(h\cdot m)(a)=h(ma)$, $h\in  \Hom_\set(A,B)$. The set of isomorphism classes of left filtered $M$-sets will be denoted by $_M\mathsf{F}$.
  
The following well-known fact \cite[p.24]{m} is a very useful tool for checking whether a given $M$-set is filtered.

\begin{Le}\label{Fcond} A left $M$-set $A$ is filtered if and only if the following three conditions hold:
\begin{itemize}
\item[(F1)] $A\not=\emptyset$.
\item[(F2)] Let $m_1,m_2\in M$, $a\in A$ and assume they satisfy the condition
$$m_1a=m_2a.$$
There exist $m\in M$ and $\tilde{a}\in A$ such that $m\tilde{a}=a$ and $m_1m=m_2m$.
\item[(F3)] Let $a_1,a_2\in A$. There are $m_1,m_2\in M$ and $a\in A$ such that $m_1a=a_1$ and $m_2a=a_2$.
\end{itemize}
\end{Le}

We give the following basic examples to help illustrate the category of topos points of $\set_M$.

\begin{Exm} Let $M=\{1\}$ be the trivial monoid. Our topos is just the category of sets in this case. As such, condition (F2) always holds. Any filtered $M$-set $A$ must have at least one element by (F1), and by (F3), it can have at most one. Thus, $\Pts(\set)$ is equivalent to the category with one object and one morphism. So, $\set$ has a unique point, up to a canonical isomorphism.
\end{Exm}

\begin{Exm}\label{t-233} Let $G$ be a group and $A$ a filtered $G$-set. We have the $G$-set homomorphism
$$\lambda_a:G\rightarrow A, \ \ \ \lambda_a(g)=ga$$
for any $a\in A$. Assume $\lambda_a(g)=\lambda_a(h)$. That is, $ga=ha$. It follows from (F2) by taking $m_1=g$ and $m_2=h$, that $gm=hm$ for some $m\in G$. We obtain $g=h$ since $G$ is a group and hence, $\lambda_a$ is injective.

Take $x\in A$ and use condition (F3) for $a_1=a$ and $a_2=x$. There are $m_1,m_2\in G$ and $a_0\in B$, such that $m_1a_0=a$ and $m_2a_0=x$. Thus, $x=m_2m_1^{-1}b=\lambda_b(m_2m_1^{-1})$ and  $\lambda_b$ is surjective. We have shown that the category $\Pts(\set_G)$ is equivalent to $G$, considered as a one object category.
\end{Exm}

\begin{Exm}\label{ex1} Let us consider $M$ as a left $M$-set. The action is obviously given by the multiplication in $M$. Condition (F1) holds since $1\in M$. To see that (F2) is satisfied, take $m=a$ and $\tilde{a}=1$. Similarly, we can take $m_1=a_1,m_2=a_2$ and $a=1$ for (F3). We call this the \emph{trivial point} of $\set_M$.
\end{Exm}

\begin{Exm} Let $M_1$ and $M_2$ be monoids. The topos $\set_{M_1\times M_2}$ is the 2-categorical product  (often called 2-product or pseudo-product) of $\set_{M_1}$ and $\set_{M_2}$ in the $2$-category of topoi, see \cite[p.419, Exercise VII.14]{mm}. Hence,
$$\Pts(\set_{M_1\times M_2})\simeq \Pts(\set_{M_1})\times\Pts(\set_{M_2})$$
and therefore
$${\sf F}_{M_1\times M_2}\simeq {\sf F}_{M_1}\times {\sf F}_{M_2}.$$
\end{Exm}

\begin{Le}\label{sk_ex} Let $f:M\rightarrow N$ be a homomorphism of monoids. If $A$ is a filtered $M$-set, then
$$f_!(A):=A\otimes_MN$$
is a filtered $N$-set. Here, $N$ is considered as an $M$-set via the homomorphism $f$. This yields a functor
$$f_!:\, \F_M\rightarrow\,\F_N,$$
and hence a functor
$$f_*:\Pts(\set_M)\rightarrow \Pts(\set_N).$$
\end{Le}

\begin{proof} Take $X$ to be an $N$-set. We need to show that the functor
$$X\mapsto f_!(A)\otimes_N X$$
commutes with finite limits. But this functor is the same as 
$$X\mapsto (A\otimes_M N)\otimes_NX \simeq A\otimes_M f^*(X).$$
Both $f^*:_N\set\rightarrow _M\set$ and $(-)\otimes_M A:_M\set\rightarrow \set$ commute with finite limits. As such, the composite functor also commutes.
\end{proof}

\begin{Co}\label{s1} Let $f:M\rightarrow N$ be an injective monoid homomorphism. Any filtered $M$-set is isomorphic to an $M$-subset of a filtered $N$-set.
\end{Co}

\begin{proof} Let $A$ be a filtered $M$-set. By Lemma \ref{sk_ex}, $A\otimes_MN$ is a filtered $N$-set. The functor $A\otimes_M(-)$ respects injective morphisms since $A$ is filtered. We apply it to the inclusion $M\rightarrow N$ to obtain that
$$A\simeq A\otimes_MM\longrightarrow A\otimes_MN$$
is injective and we are done.
\end{proof}

\subsection{The case of commutative monoids}

From here on onwards, we will assume that $M$ is a commutative monoid. As such, we will no longer make any distinction between left and right $M$-sets. The category of $M$-sets will be written as $\set_M$ and the action of $M$ on $A$ will be written as $ma$.

\begin{Pro}\label{ftp} Let $M$ be a monoid and $A,B$ filtered $M$-sets. Then $A\otimes_MB$ is a filtered $M$-set as well.
\end{Pro}

\begin{proof} 
Consider the commutative diagram of topoi:
$$\xymatrix{ \set_M \ar[rr]^{(-)\otimes_M B}\ar[drr]_{(-)\otimes_M (A\otimes_M B) \quad} & & \set_M\ar[d]^{(-)\otimes_MA} \\ 
			 & & \set. }$$
The composition $(-)\otimes_M(A\otimes_MB)$ commutes with finite limits and all colimits since both functors $(-)\otimes_MA$ and $(-)\otimes_MB$ do.
\end{proof}

\subsubsection{From prime ideals to topos points}

In this subsection, we will construct a map $\gamma_M:\Spec(M)\rightarrow \sF_M$ from the set of prime ideals of $M$ to the set of isomorphism classes of points of the topos $\set_M$.

\begin{Le} \label{221} Let $\p$ be a prime ideal of $M$. The localisation $M_\p$ is a filtered $M$-set, where  $M_\p$ is considered as an $M$-set via the canonical monoid homomorphism $M\rightarrow M_\p$.
\end{Le}

\begin{proof} The set $M_\p$ is not empty since $1\in M$. Hence, (F1) holds. Assume $m_1a=m_2a$ for $m_1,m_2\in M$ and $a\in M_\p$. We can write $a=\frac{m_3}{s}$, where $m_3\in M$ and $s\in M\setminus\p$. There is an element $t\nin \p$, such that $m_1m_3st=m_2m_3st$. This is true since $\frac{m_1m_3}{s}=\frac{m_2m_3}{s}$. Take $m=m_3st$ and $\tilde{a}=\frac{1}{s^2t}$. We see that condition (F2) holds. 

To show (F3), take two elements $a_1=\frac{m_1'}{s_1}$ and $a_2=\frac{m_2'}{s_2}$ in $M_\p$. Put $a=\frac{1}{s_1s_2}$, $m_1=m_1's_2$ and $m_2=m_2's_1$. We see that $m_1a=a_1$ and $m_2a=a_2$, proving the assertion.
\end{proof}

This lemma shows that we have a map
$$\gamma_M:\Spec(M)\rightarrow \sF_M$$
from the prime ideals of a commutative monoid $M$ to the isomorphism classes of the topos points of $\set_M$. As previously shown \cite{p1} however, $\Spec(M)$ is in fact a lattice. As such, it has two intertwined operations. On their own, both of them make it into a commutative monoid. One is simply the union (note that in the monoid world, the union of ideals is again an ideal) and the other is defined as follows: For two prime ideals $\p$ and $\q$, define $\p\Cap\q$ to be the union of all the prime ideals contained in both $\p$ and $\q$. That is,
\begin{equation}\label{pq} \p\Cap\q:=\bigcup_i\p_i \ | \ \p_i\subseteq\p\cap\q.
\end{equation}
The combination of Lemma \ref{ftp} and Example \ref{ex1} show that $\sF_M$ is a commutative monoid. It is straightforward to see that $M_\p\otimes_MM_\q\simeq M_{\p\Cap\q}$. Hence, $\gamma_M$ is a monoid homomorphism. By ii) Lemma \ref{212}, $M_\p$ and $M_\q$  are isomorphic as $M$-sets if and only if $\p=\q$. We sum up our discussion in the following proposition.

\begin{Pro}\label{inj} Let $M$ be a commutative monoid. There exists an injective homomorphism of commutative monoids
$$\gamma_M:\Spec(M)\rightarrow \sF_M,$$
given by $\p\mapsto M_\p$.
\end{Pro}

\subsubsection{From topos points to prime ideals}

The aim of this subsection is to prove that the homomorphism $\gamma_M$, defined above, is an isomorphism if $M/M^\times$ is finitely generated. Here, $M^\times$ denotes the subgroup of invertible elements. Recall that for an $M$-set $A$ and an element $m\in M$, we defined the homomorphism $l_m:A\rightarrow A$ by $l_m(a)=ma$.

\begin{Le}\label{212} Let $A$ be an $M$-set.
\begin{itemize} 
\item[i)] The set $\p_A$, or simply $\p$ when there is no ambiguity, defined by
$$\p=\{m\in M \ | \ l_m:A\rightarrow A \ {\rm is \ not \ an \ isomorphism} \},$$
is a prime ideal of $M$. Moreover, the action of $M$ on $A$ extends to an action of  $M_\p$ on $A$. Hence, $A$ is an $M_\p$-set.
\item [ii)] If $A=M_\p$, then $\p_A=\p$.
\end{itemize}
\end{Le}

\begin{proof} i) Take $m\in\p$ and $x\in M$. Assume $l_{xm}$ is an isomorphism and let $g:A\rightarrow A$ be the inverse of $l_{xm}$. Then $l_x\circ g=g\circ l_x$ since $g$ is a map of $M$-sets and $g\circ l_{xm}=\id=l_{xm}\circ g$. We have
$$l_m\circ l_x\circ g=l_{mx}\circ g=l_{xm}\circ g=\id,$$
and
$$l_x\circ g \circ l_m=g\circ l_x\circ l_m=g\circ l_{xm}=\id$$
since $l_{xm}=l_x\circ l_m$. Hence, $l_m$ is an isomorphism, which contradicts our assumption. This proves that $\p$ is an ideal.

To see that it is prime, take $m,n\nin \p$. The composition $l_{mn}=l_m\circ l_n$ is an isomorphism since both $l_m$ and $l_n$ are isomorphisms. Thus, $mn\nin \p$.

For the last statement, take an element $a\in A$ and $x=\frac{m}{s}\in M_\p$. The map $l_s$ is an isomorphism since $s\nin \p$. Put
$$xa:=ml_s^{-1}(a).$$
It is straightforward to check that this defines an action of $M_\p$ on $A$.

ii) Take $s\nin \p$. The map $l_s$ is an isomorphism. Hence, $s\nin \p_A$ and we see that $\p_A\subseteq \p$. Assume, there exists $x\in \p$ such that $l_x$ is an isomorphism. Then $1\in \im(l_x)$. So, $x\cdot \frac{y}{s}=1$ in $M_\p$ for some $y\in M$ and $s\nin \p$. It follows that $xyt\nin \p$ for some $t\nin \p$. We see that $x\nin \p$ since $\p$ is prime. This contradiction implies the result.
\end{proof}

\subsubsection{Conservative $M$-sets}

We call an $M$-set $A$ \emph{conservative}, provided any $m\in M$ for which $l_m:A\rightarrow A$ is an isomorphism, is an invertible element of $M$.

\begin{Le} \label{231} Let $A$ be an $M$-set. Then $A$ is conservative as an $M_\p$-set, where $\p=\p_A$ is the prime ideal defined in Lemma \ref{212}.
\end{Le}

\begin{proof} Any element $x\in M_\p$ can be written as $x=\frac{m}{s}$, where $s\nin \p$. We have  $l_x=l_s^{-1}\circ l_m$. Assume $l_x:A\rightarrow A$ is an isomorphism. Then $l_m=l_s\circ l_x$ is also an isomorphism. Thus, $m\nin \p$ and therefore, $x$ is invertible in $M_\p$.
\end{proof}
 
\subsubsection{Source} 

Let $A$ be an $M$-set. We call an element $a\in A$ a \emph{source}, provided $a=mb$ implies that $l_m:A\rightarrow A$ is an isomorphism. Here, $m\in M$ and $b\in A$ .

\begin{Le}\label{fp} Let $\p$ be a prime ideal of $M$ and assume that $M_\p$ is a conservative $M$-set. Then $M_\p\simeq M$ and $1$ is a source of $M_\p$.
\end{Le}

\begin{proof} Take $s\in M\setminus\p$. The homomorphism $l_s:M_\p\rightarrow M_\p$, $x\mapsto sx$ is an isomorphism since $s$ is invertible in $M_\p$. The definition of conservativeness implies that $s$ is already invertible in $M$ and as such, $M\simeq M_\p$. It is now clear that the unit, indeed every invertible element, is a source.
\end{proof}

\begin{Le}\label{141a} Let $A$ be a filtered and conservative $M$-set and take an element $a\in A$. The map
$$\lambda_a:M\rightarrow A, \ \ \lambda_a (m)=ma$$
is an isomorphism if $a$ is a source of $A$.
\end{Le}

\begin{proof} Take $b\in A$. By the property (F3), there are $c\in A$, $m_1,m_2\in M$, such that $a=m_1c$ and $b=m_2c$. It follows that $m_1$ is invertible since $A$ is a conservative $M$-set and $a$ is a source. Hence, $b=m_2m_1^{-1}a$ and $\lambda_a$ is surjective.
	
Assume $\lambda_a(m_1)=\lambda_a(m_2)$, that is $m_1a=m_2a$. By (F2), there are elements $b\in A$ and $m\in M$, such that $a=mb$ and $m_1m=m_2m$. Since $a$ is a source and $A$ conservative, $m$ must be invertible and thus, $m_1=m_2$. This shows that $\lambda_a$ is injective and hence, an isomorphism.
\end{proof}

\begin{Le} \label{surj} Let $A$ be a filtered $M$-set and  $m\in M$. The map $l_m:A\rightarrow A$, given by $a\mapsto ma$, is an isomorphism if it is surjective.
\end{Le}

\begin{proof} Let $a_1,a_2\in A$ and assume $ma_1=ma_2$ holds. Property (F3) says that there exist $a\in A$ and $m_1,m_2\in M$, such that $m_1a=a_1, m_2a=a_2$. Thus, $mm_1a=mm_2a$. By (F2), there exist an element $b\in A$ and $m_3\in M$, such that $a=m_3b$ and $mm_1m_3=mm_2m_3$. There exist $c\in A$ such that $b=mc$ since $l_m$ is surjective. We have
$$a_1=m_1a=m_1m_3b=m_1m_3mc=m_2m_3mc=m_2m_3b=m_2a=a_2.$$
Hence, $l_m$ is injective and as such, an isomorphism.
\end{proof}

We will say that a commutative monoid $M$ is \emph{almost finitely generated} if there are is a finite number of non-invertible elements $c_1,\cdots,c_k\in M\setminus M^\times$, such that any $m\in M$ can be written as a product $m=ge_1^{\alpha_1}\cdots e_k^{\alpha_k}$. Here, $g\in M^\times$ is invertible and $\alpha_1,\cdots,\alpha_k\in \mathbb{N}$ are natural numbers. The elements $e_1,\cdots, e_k$ are called \emph{almost generators}. Clearly, this class encompasses any group, as well as any finitely generated monoid. Moreover, if $M$ is almost finitely generated and $\p$ is a prime ideal of $M$, then $M_\p$ is also almost finitely generated. Indeed, this is true for all localisations, since every localisation in a monoid is isomorphic to a localisation with a prime ideal, see \cite{chww}.

An other way of saying that $M$ is almost finitely generated is that $M/M^\times$ is finitely generated. 

\begin{Le} \label{zn} Let $M$ be an almost finitely generated monoid and $A$ a filtered and conservative $M$-set. There exists an element $a\in A$, such that the canonical map $\lambda_a:M\rightarrow A$, $m\mapsto ma$  is an isomorphism of $M$-sets.
\end{Le}

\begin{proof} Let $e_1,\cdots, e_k$ be almost generators of $M$. The map $l_{e_i}:A\rightarrow A$ is not an isomorphism for any $i=1,\cdots,k$ since $A$ is conservative. It follows from Lemma \ref{surj} that the maps $l_{e_i}$ are not surjective. Hence, there is an element $a_i\in A$ for every $i$, such that $a_i$ is not in the image of $l_{e_i}$. This means that the equations $e_ix=a_i$ have no solutions in $A$. We know from (F3) that there exists an element $a\in A$ and a finite collection $m_1,\cdots, m_k\in M$, such that
$$a_i=m_ia, \ \ \ i=1,\cdots, k.$$
We claim that $a$ is a source of $A$. Let $a=mb$, where $b\in A$ and $m\in M$. Our assumption that $M$ is almost finitely generated implies that the element $m$ can be written in the form
$$m=ge_1^{\alpha_1}\cdots e_k^{\alpha_k}, \ \ \ \alpha_i\in\N.$$
Assume there exists an $i\in\{1,\cdots,k\}$, such that $\alpha_i>0$. We have $m=e_im'$ and thus,
$$a_i=m_ia=m_imb=m_i(e_im')b=e_ix,$$
where $x=m_im'b$. This contradicts the assumption that $a_i=e_ix$ has no solution. As such, $\alpha_1=\cdots =\alpha_k=0$ and $m=g$ is invertible in $M$. It follows that $a$ is a source. Since $M$ is conservative, $\lambda_a:M\rightarrow A$ is an isomorphism by Lemma \ref{141a}.
\end{proof}

We are now in position to prove our main theorem of this section.

\begin{Th}\label{thm} Let $M$ be an almost finitely generated monoid. There exists an isomorphism of commutative monoids
$$\gamma_M:\Spec(M)\rightarrow \sF_M,$$
given by $p\mapsto M_p$. In particular, there exist a prime ideal $\p$ for any point $f:\set\rightarrow \set_M$ of the topos $\set_M$, such that the functors 
$$f^*:\set_M\rightarrow \set \  \ {\rm and} \ \ (-)_{\p}:\set_M\rightarrow \set$$
are isomorphic.
\end{Th}

\begin{proof} Let $A$ be a filtered $M$-set and consider the prime ideal $\p=\p_A$ defined in Lemma \ref{221}. Then $A$ becomes an $M_\p$-set, which is conservative and filtered over $M_\p$. Hence, Lemma \ref{zn} implies that $A$ is isomorphic to $M_\p$ and we are done.
\end{proof}

\subsubsection{$\Delta(P)$: Ordering of prime ideals as topos points} Let $\ta$ be a topos and $P=(P^\bullet,P_\bullet)\in\Pts(\ta)$ a point. Let $\End(P)$ be the monoid of endomorphisms of $P$, which is the monoid of natural transformations of the inverse image functor $P^\bullet$.

Let $X$ be an object of $\ta$, $x\in P^\bullet(X)$ and $\alpha\in\End(P)$. The set $P^\bullet(X)$ has a natural left $\End(P)$-set structure given by
$$\alpha\cdot x:=\alpha(X)(x).$$
It follows that the functor $P$ has a canonical decomposition
$$\xymatrix{ \set\ar@/_1.5pc/[rr]_P\ar[r]^{\mathcal{U}} & \set_{\End(P)}\ar[r]^{\hat{P}} & \ta. }$$
Here, $\mathcal{U}$ is the trivial topos point of $\set_{\End(P)}$, whose inverse image functor is the forgetful functor and $\hat{P}$ is a geometric morphism. Let $Q:\set\rightarrow\set_{\End(P)}$ be a topos point of $\set_{\End(P)}$. The composition $\hat{P}\circ Q$ defines a topos point of $\ta$. The subcategory of $\Pts(\ta)$ obtained in this way will be denoted by $\Delta(P)$. These are the topos points of $\ta$ that factor through $\set_{\End(P)}$. We clearly have $P\in\Delta(P)$. This can be thought of as the `\emph{canonical neighbourhood}' of the topos point $P$ in the `space' of all points of $\ta$.

We aim to show that this agrees with the Zariski topology for almost finitely generated commutative monoids.

Let $M$ be a commutative monoid and $\p,\q\in\Spec(M)$ its prime ideals. Denote by $P,Q:\set\rightarrow \set_{M_\p}$ the topos points associated to these prime ideals. In other words, we have $P^\bullet=-\otimes_MM_\p$ and $Q^\bullet=-\otimes_MM_\q$.

\begin{Le}\label{x4} Let $F:\set\rightarrow \set_{M_\p}$ be a topos point such that the diagram
$$\xymatrix{ \set_M\ar[r]^{-\otimes_MM_\p}\ar[dr]_{-\otimes_MM_\q} & \set_{M_\p}\ar[r]^{F^\bullet} & U \\
			 & \set_{M_\q}\ar[ur]_{U} & }$$
commutes. Here, $U$ denotes the forgetful functor. We have $F^\bullet=-\otimes_{M_\p}M_\q$, where we identify $\q=\q M_\p\in\Spec(M_\p)$.
\end{Le}

\begin{proof} We know that the inverse image $F^\bullet$ of $F$ is given by $-\otimes_{M_\p}A$, where $A$ is a filtered $M_\p$-set, and simultaneously, an $\End(F)$-set. This yields the following commutative diagram
	$$\xymatrix{ \set_M\ar[r]^{-\otimes_MM_\p}\ar[dr]_{-\otimes_MM_\q} & \set_{M_\p}\ar[r]^{-\otimes_{M_\p}A} & \set_{\End(F)}\ar[r]^U & \set \\
		& \set_{M_\q}\ar[rru]_U. }$$
	Here, $U$ denotes the forgetful functor. Commutativity implies that
	$$U((-\otimes_MM_\p)\otimes_{M_\p}A)\simeq U(-\otimes_MM_\q).$$
	In particular
	$$(M\otimes_MM_\p)\otimes_{M_\p}A\simeq M\otimes_MM_\q,$$
	and hence, $A\simeq M_\q$.
\end{proof}

This Lemma has the following direct corollary.

\begin{Co}\label{x5} We have $Q\in\Delta(P)$ if and only if $\q\subseteq \p$.
\end{Co}

We will also mention the following little fact, thought it is not needed for the purposes of this paper.

\begin{Le} We have $Q\in\Delta(P)$ if and only if $\Delta(Q)\subseteq \Delta(P)$.
\end{Le}

\begin{proof} One side is clear. Assume $Q\in\Delta(P)$ and let $Q_1\in\Delta(Q)$. There exists a geometric morphism $G:\set\rightarrow \set_{M_\q}$ such that $G^\bullet\circ (-\otimes_MM_\q)=Q_1^\bullet$. Since $Q\in\Delta(P)$, we know from Lemma \ref{x4} and Corollary \ref{x5} that there exists a topos point $F:\set\rightarrow \set_{M_\p}$, with $F^\bullet=-\otimes_{M_\p}M_\q$. In particular, we have $F^\bullet\circ (-\otimes_MM_\p)\simeq-\otimes_MM_\q$. Hence,
$$Q_1^\bullet=G^\bullet\circ(-\otimes_MM_\q)=(G^\bullet\circ F^\bullet)\circ(-\otimes_MM_\p),$$
showing that $Q_1\in\Delta(P)$.
\end{proof}

Let $M$ be a commutative monoid. Consider $M^{sl}=M/\sim$, where $\sim$ is the congruence generated by $m^2=m$ for every $m\in M$. Explicitly, two elements $a,b\in M$ become equivalent under this congruence if there are elements $u,v\in M$ and integers $m,n\in \N$, such that $a^n=bu$ and $av=b^m$. This is the universal semilattice associated to $M$. We have the following result (see \cite{p1} for the first half and Proposition \ref{inj} for the second half):

\begin{Pro}\label{55125} Let $M$ be a commutative monoid. There are injective homomorphisms of commutative monoids
$$M^{sl}\xrightarrow{\alpha_M}\Spec(M)\xrightarrow{\gamma_M}{\sf F}_M,$$
which are isomorphisms when $M$ is finitely generated. The monoid structures are induced by the quotient, the operation $\Cap$ and the tensor product respectively.
\end{Pro}

We note that $\Spec$ is a contravariant functor, whereas $sl$ and $F$ are both covariant functors. The composite map $\gamma_M\circ \alpha_M$ sends an element $[f]\in M^{sl}$ (which can be thought of as a class of an element $f\in M$) to the point of the topos of $M$-sets corresponding to the filtered $M$-set $M_f$. Based on this fact, one easily sees that $\gamma_M\circ \alpha_M$ defines a natural transformation from the functor $M\mapsto M^{sl}$ to the functor $M\mapsto {\sf F}_M$.

We will see in Section \ref{algam} that when $M$ is not finitely generated, neither $\alpha_M$ nor $\gamma_M$ are isomorphisms in general.

\section{Graded monoids and topoi}\label{060804}

From the point of view of monoid schemes, commutative monoids are the dual objects of affine monoid schemes. One can hope that the results of the previous section are just the affine versions of far more general results about monoid schemes. 

We will see in the next section that this is the case, at least for monoid schemes which are open subschemes of `projective monoid schemes'. The latter is based on the theory of graded monoids, much like in classical algebraic geometry. Hence, this section aims at collecting some basic results on graded monoids.

\subsection{Graded sets}\label{grsets}

Let $\{X_i\}_{i\in \Z}$ be a set indexed by the integers. This is called a \emph{graded set} and is dented by $X_*$. The disjoint union ${\sf Tot}(X_*):=\coprod _{i\in \mathbb{Z}}X_i$ is called the  \emph{total} or \emph{underlying set} of the graded set $X_*$.

A \emph{graded map} between $X_*$ and $Y_*$, denoted by $f_*:X_*\rightarrow Y_*$, is given by a family of maps $\{f_i:X_i\rightarrow Y_i\}_{i\in \Z}$. Any such family induces a well defined map ${\sf Tot}(f_*):{\sf Tot}(X_*)\rightarrow {\sf Tot}(Y_*)$ of underlying sets, which sends an element $x\in X_i$ to $f_i(x)$.

The category of graded sets and graded maps is denoted by $_{\sf gr}\set$. It is the product of $\Z$-copies of $\set$. Clearly, ${\sf Tot}: _{\sf gr}\set\rightarrow \set$ is the forgetful functor.

The category $_{\sf gr}\set$  is a topos. The limits and colimits are computed degreewise. Observe that the forgetful functor $\sf Tot$ preserves colimits. But it does not preserves finite products. The terminal object of $_{\sf gr}\set$ is denoted by ${\bf 1}_*$. The $i$-th set ${\bf 1}_i$ is the one element set, say $\{i\}$, for any $i\in \mathbb{Z}$.  It ia clear that ${\sf Tot}({\bf 1}_*)=\mathbb{Z}$.

\subsubsection{Graded product}

Let $X_*$ and $Y_*$ be graded sets. We define the graded product $X_*\boxtimes Y_*$ by
$$(X_*\boxtimes Y_*)_i=\coprod_{j+k=i}X_j\times Y_k.$$
We see that $\boxtimes$ defines a symmetric monoidal structure on the category $_{\sf gr}\set$. The unit object for $\boxtimes$ is $\mathbb{I}_*$, where
$$\mathbb{I}_i=\begin{cases} \emptyset, & i\not=0,\\ \{1\}, & i=0.\end{cases}$$
One easily checks that
$${\sf Tot}(X_*\boxtimes Y_*)\simeq {\sf Tot}(X_*)\times {\sf Tot}(Y_*)$$
and ${\sf Tot}(\mathbb{I}_*)\simeq \{1\}$. In other words, the functor $\sf Tot$ is symmetric monoidal.

\subsubsection{Shifting}\label{shifting}

The category $_{\sf gr}\set$ posses a special class of endofunctors called \emph{shiftings}. Let $X_*$ be a graded set and $i$ an integer. Denote by $X_*(i)$ the graded set given by
$$(X_*(i))_j=X_{i+j}.$$
This yields endofunctors $S_i:$ $_{\sf gr}\set\rightarrow$ $_{\sf gr}\set$, given by $X_*\mapsto X_*(i)$ for all $i\in \mathbb{Z}$. Observe that ${\sf Tot}(X_*(i))={\sf Tot}(X_*)$.

\subsection{Graded monoids}

A monoid object in the symmetric monoidal category $(_{\sf gr}\set,\boxtimes,\mathbb{I})$ is called a \emph{graded monoid}. In other words, a graded monoid is a graded set $M_*$, together with a distinguished element $1\in M_0$ and product maps $M_i\times M_j\rightarrow M_{i+j}$, $i,j\in \Z$. Further, these must satisfy the associativity and the unit condition. It follows that ${\sf Tot}(M_*)$ is a monoid, called the \emph{underlying monoid} of the graded monoid $M_*$. Let $M'_*$ and $M''_*$ be graded monoids. Their graded product $M_*=M_*'\boxtimes M_*''$ has the structure of a graded monoid given by
$$(M'_*\boxtimes M''_*)\boxtimes (M'_*\boxtimes M''_*)\simeq (M'_*\boxtimes M'_*)\boxtimes (M''_*\boxtimes M''_*)\rightarrow(M'_*\boxtimes M''_*).$$
The last map is induced by the graded monoid structures on $M_*'$ and $M_*''$.

Let $M_*$ and $M_*'$ be graded monoids. A graded map $f_*:M_*\rightarrow M_*'$ is called a \emph{homomorphism of graded monoids}, or \emph{graded morphism} for short, if one has $f_{i+j}(xy)=f_i(x)f_j(y)$ for all $i,j\in \Z$, $x\in M_i$ and $y\in M_j$. It is clear that $M_*\mapsto {\sf Tot}(M_*)$ defines a covariant functor from the category of graded monoids to the category of monoids.
 
Let $M_*$ be a graded monoid and $m\in {\sf Tot}(M_*)$. Denote by $|m|$ the unique integer for which $m\in M_{|m|}$. This gives us a map $|-|:M_*\rightarrow \Z$, called the \emph{degree map}. We have $|m_1m_2|=|m_1|+|m_2|$, making it a monoid homomorphism.

Conversely, let $M$ be a monoid and $\mu:M\rightarrow \Z$ a monoid homomorphism. This allows us to see $M$ as the underlying monoid of a graded monoid $M_*$ for which $M_i=\mu^{-1}(i)$.

We can regard $\Z$ and $\N$ as the underlying monoids of graded monoids by considering the identity and the inclusion map respectively. The graded product $\boxtimes$ allows us to do the same for $\Z^d$ and $\N^d$, $d\geq 1$ as well.

A graded map $f_*:M_*\rightarrow M_*'$ is a \emph{morphism of graded monoids} if and only if the diagram
$$\xymatrix{ M_*\ar[rr]^f\ar[dr]_{|-|} & & M_*'\ar[dl]^{|-|}
			 \\ & \Z & }$$
commutes.

It is quite useful to consider a monoid $M$ as a one object category. Covariant (resp. contravariant) functors over this category are nothing but left (resp. right) $M$-sets. One can define a category $\cat(M_*)$ in a similar vain by starting with a graded monoid $M_*$. Objects of $\cat(M_*)$ are integers. The set of homomorphism $\Hom_{\cat(M_*)}(i,j)$ between two objects (integers) $i,j\in \mathbb{Z}$ is $M_{j-i}$. In particular, $\Hom_{\cat(M_*)}(i,j)=\emptyset$ if $i>j$. The composition is induced by the monoid structure of $M_*$.

\subsection{Graded $M_*$-sets}

Let $M_*$ be a graded monoid. A \emph{graded left $M_*$-set} is a graded set $X_*$, together with a morphism $M_*\boxtimes X_*\rightarrow X_*$ satisfying the obvious axioms. Equivalently, it is a collection of maps $M_i \times X_j\rightarrow X_{i+j}, \ (m,x)\mapsto mx$ satisfying the properties:
$$1x=x \quad {\rm and } \quad (mn)x=m(nx) \quad \textnormal{for any} \quad x\in X_i, \ m\in M_k, \ n\in M_j.$$
Graded right $M_*$-sets are similarly defined.  

Let $X_*$ and $Y_*$ be graded left $M_*$-sets. A morphism $f_*:X_*\rightarrow Y_*$ of graded sets is called a \emph{morphism of graded left $M_*$-sets} if $f_{i+j}(mx)=mf_j(x)$ for any $m\in M_i$, $x\in X_j$.  We let $_{M_*}\set$ be the category of graded left $M_*$-sets and graded morphisms of left $M_*$-sets. Accordingly, $ \Hom_{_{{\rm gr}{M_*}}\set}(X_*,Y_*)$ denotes the set of all graded morphisms of left $M_*$-sets from $X_*$ to $Y_*$. Since ${\sf Tot}$ is a symmetric monoidal functor, we see that the set ${\sf Tot}(X_*)$ has a natural left ${\sf Tot}(M_*)$-set structure for any graded left $M_*$-set $X_*$. This gives us a functor
$$_{{\rm gr}M_*}\set \rightarrow _{{\sf Tot}(M_*)}\set$$
from graded left $M_*$-sets to left ${\sf Tot}(M_*)$-sets. Likewise also for graded right $M_*$-sets.

\begin{Le}\label{catMs} Assume $M_*$ is a graded monoid. 
\begin{itemize}
\item [i)] Let $X_*$ be a graded left $M_*$-set. Then
$i\mapsto X_i$ extends to a covariant functor
$$\f(X_*):\cat(M_*)\rightarrow \set.$$
Moreover, this assignment is an equivalence of categories:
$$\f:_{{\rm gr}M_*}\set \xrightarrow{\simeq} {\sf Funct}( \cat(M_*), \set).$$
\item[ii)] Let $X_*$ be a graded right $M_*$-set. Then
$i\mapsto X_{-i}$ extends to a contravariant functor
$$\f^{op}(X_*):\cat(M_*) \rightarrow \set.$$
Moreover, this assignment is an equivalence of categories:
$$\f^{op}:\set_{{\rm gr}M_*} \xrightarrow{\simeq} {\sf Funct}(\cat(M_*)^{op}, \set).$$
\end{itemize}
\end{Le}

\begin{proof} i) We already described $\f(X_*)$ on objects: 
$$(\f(X))(i)=X_i.$$
It act on morphisms of $\cat(M_*)$ as follows: Let $m\in M_{j-i}$. The map
$$(\f(X))(m):(\f(X))(i)=X_i\rightarrow X_j=(\f(X))(j)$$
is given by
$$X_i\ni x\mapsto mx\in X_j.$$
Conversely, take a  covariant functor $F:\cat(M_*)\rightarrow \set$ and define $\fx(F)_*$ to be the graded set $(F(i))_{i\in \Z}$. Moreover, the product $mx$ is defined as $F(m)(x)\in F(j)$ for $m\in M_{j-i}$ and $x\in F(i)$. We obtain the functor
$$\fx: {\sf Funct}( \cat(M_*), \set) \rightarrow {} _{{\rm gr}M_*}\set,$$
which is the inverse of $\f$.

ii) This is quite similar to i). We just remark that the right multiplication by $r\in M_{j-i}$ defines a map $X_{-j}\rightarrow X_{-i}$ when considered as a morphism $i\rightarrow j$.
\end{proof}

The category $\set_{{\rm gr}M_*}$ is a topos, as it is a category of presheaves. In particular, it has limits and colimits, which are given componentwise. For example, if $X_*$ and $Y_*$ are graded right $M_*$-sets, then
$$(X_*\coprod Y_*)_i=X_i\coprod Y_i, \ \ (X_*\times Y_*)_i=X_i\times Y_i.$$
We see that the obvious forgetful functor
$${\sf Tot}: \set_{{\rm gr}M_*} \rightarrow \set_{M_*}$$
respects the coproduct. It does, however, not preserve the product in general.

\subsection{The bifunctor $\boxtimes_{M_*}$}

Let $M_*$ be a graded monoid, $X_*$ a graded right $M_*$-set and $Y_*$ a graded left $M_*$-set. Define $X_* \boxtimes _{M_*} Y_*$ to be the coequaliser of the diagram
$$\xymatrix{ X_* \boxtimes M_*\boxtimes Y_* \ar@<-.5ex>[r] \ar@<.5ex>[r] & X_*\boxtimes Y_*. }$$ 
Here, the first arrow is induced by the right action of $M_*$ on $X_*$. The second arrow is induced by the left action of $M_*$ on $Y_*$. The $0$-th component of  $X_* \boxtimes _{M_*} Y_*$ is denoted by $X_*\odot_{M_*} Y_*$. Thus, $X_*\odot_{M_*} Y_*$ is the set of equivalence classes of pairs $(x,y)$, where $x\in X_i$ and $y\in Y_{-i}$. The equivalence relation is generated by $(um,v)\sim (u,mv)$. Here, $u\in X_{i}$, $m\in M_j$ and $v\in Y_{-i-j}$.  

\begin{Le} Let $M_*$, $X_*$ and $Y_*$ be as above. One has the following isomorphism of sets:
$$i) \hspace{9em} {\sf Tot}(X_* \boxtimes _{M_*} Y_*)={\sf Tot}(X_*)\otimes_{{\sf Tot}(M_*)}{\sf Tot}(Y_*).\hspace{9em}$$
$$ii) \hspace{11em} X_*\odot_{M_*} Y_*\simeq \f^{o}(X_*)\otimes_{\cat(M_*)}\f(Y_*).\hspace{11em}$$

\noindent The right hand side of the second equivalence is the tensor product in the category of functors \cite[p.256]{mm}.
\end{Le}

\begin{De} A graded left $M_*$-set $X_*$ is called \emph{filtered} if the following conditions hold:
\begin{itemize}
\item [i)] There is an integer $i$ such that $X_i\not=\emptyset$.
\item [ii)] Let $a\in X_i$ and $b\in X_j$. There exists an integer $k$ and elements $u\in M_{i-k}$, $v\in M_{j-k}$, $c\in M_k$, such that $uc=a$ and $vc=b$.
\item [iii)] Let $a\in X(i)$ and $u,v\in M_j$ such that $ua=va$. There is an integer $k$ and elements $b\in X(k)$, $w\in M_{i-k}$, such that $uw=vw$ and $wb=a$. 
\end{itemize}
\end{De}

In other words, $X_*$ is filtered if and only if $\f(X)$ is a filtered functor \cite[p386]{mm}. 

\begin{Le} Assume $X_*$ is a filtered graded left $M_*$-set. The graded left $M_*$-set $X_*(i)$ is filtered for any integer $i$.
\end{Le}

The classical version of Diaconescu theorem implies the following result.

\begin{Pro} The category of points of the topos $\set_{{\rm gr}M}$ is equivalent to the category of graded filtered left $M_*$-sets. The inverse image functor $P_{X_*}:\set_{{\rm gr}{M_*}}\rightarrow \set$ of the point corresponding to a graded filtered left $M_*$-set $X_*$ is given by $(-)\odot _{M_*}X_*$.
\end{Pro}

\subsection{Localisations of a commutative graded monoid}

We will henceforth assume that $M_*$ is commutative and will no longer make any distinction between left and right $M_*$-sets.
Let $M_*$ be a graded monoid and $S$ a submonoid of ${\sf Tot}(M_*)$. One sets $S_i=S\cap M_i$, $i\in \Z$. This gives us the graded set $S_*$, which is in fact a graded submonoid of $M_*$. We can form the localisation $S^{-1}({\sf Tot}(M_*))$.

The degree map ${\sf Tot}(M_*)\rightarrow \Z$ has an extension
$$S^{-1}({\sf Tot}(M_*)) \rightarrow \Z,$$
where $$\left \mid \frac{m}{s}\right \mid=|m|-|s|.$$
Thus, $S^{-1}({\sf Tot}(M_*))$ is an underlying monoid of a graded monoid $S_*^{-1}M_*$. Here, $(S_*^{-1}M_*)_i$ is the set of all fractions $\frac{m}{s}$, $m\in M_{|m|}$, $s\in S_{|s|}$ for which $|m|=|s|+i$.

This construction has an obvious continuation to graded $M_*$-sets. If $A_*$ is a graded $M_*$-set, then $S_*^{-1}A_*$ is a graded $S_*^{-1}M_*$-set. The $i$-th degree consists of all fractions $\frac{a}{s}$, $a\in A_{|m|}$, $s\in S_{|s|}$ for which $|a|=|s|+i$.

Of special interest are the cases when $S=\{f^n \ | \ n\in \N\}$ and when $S$ is the complement of a prime ideal $\p$ of ${\sf Tot}(M_*)$. Their corresponding localisations are written as $A_{*f}$ and $A_{*\p}$. We will denote the $0$-th degree of $A_{*f}$ and $A_{*\p}$ by $A_{*(f)}$ and $A_{*(\p)}$ respectively. For example, $A_{*(f)}=\{ \frac{m}{f^k}; |m|=k|f|\}$ and $A_{*(\p)}=\{\frac{m}{r};|m|=|r|\}$. It follows from the definition that $A_{*(f)}\simeq A_*\odot _{M_*} M_{*f}$ and similarly, $A_{*(\p)}\simeq A_*\odot _{M_*} M_{*\p}$.

\subsection{Positively graded monoids}

A graded monoid $M_*$ is called \emph{positively graded} if $M_i=\emptyset$ for $i<0$. Henceforth, all graded monoids are going to be assumed to be positively graded. Denote
$$M_+=\coprod_{d>0}M_d.$$
Then $M_+$ is a prime ideal of ${\sf Tot}(M_*)= \coprod_{d\geq 0}M_d$.

Let $\m$ be an ideal of ${\sf Tot}(M_*)$ and set $\m_d= \m\cap M_d$. Then $\m_0$ is an ideal of $M_0$ and $\m=\coprod_{d\ge 0} \m_d$. We have $\m=\m_0\coprod \m_+$, where $\m_+=\coprod_{d>0}\m_d$.

We omit the proof of the following theorem as it is identical to the proof of Theorem \ref{thm}.

\begin{Th}\label{pgrprime} Let $M_*$ be a commutative and positively graded monoid. Assume that $M_0$ is a finitely generated monoid, $M_1$ is a finitely generated $M_0$-set and $M_0$ and $M_1$ generate ${\sf Tot}(M_*)$ as a monoid. Any filtered $M_*$-set $X_*$ is of the form $M_{*(\p)}(i)$. Here, $i\in\Z$ is an integer and $\p$  is a prime ideal of ${\sf Tot}(M_*)$. The corresponding point of the topos  $\set_{{\rm gr}M_*}$ is given by $A_*\mapsto (A_{*}(i))_{(\p)}$.
\end{Th}

\section{The topos of quasi-coherent sheaves over the monoid scheme $\Proj(M_*)$}\label{poproj}

Let $X$ be a monoid scheme. We have seen (Theorem \ref{432.05.03}) that the category $\Qc(X)$ of quasi-coherent sheaves over $X$ is an $s$-noetherian topos in favourable cases. The stalk at a point $x\in X$ defines a point of the topos $\Qc(X)$. We have seen that the converse is also true if $X=\Spec(M)$ is affine and $M$ is almost finitely generated. That is, any point of the topos $\Qc(X)$ is isomorphic to the stalk functor at a point of $X$.

We wish to generalise this result to the non-affine case. We will show that this is true for a large class of monoid schemes, namely projective monoid schemes.

The graded monoids in this section are all assumed to be commutative and positively graded.

\subsection{ Monoid scheme $\Proj(M_*)$}

Let $M_*$ be a positively graded monoid. We can construct the monoid scheme $\Proj(M_*)$ just like in classical algebraic geometry \cite{chww}. We recall the definition:

Denote by $V(M_+)$ the set of all prime ideals $\p$ of ${\sf Tot}(M_*)$ for which $M_+\subseteq \p$. The map
$$\Spec(M_0)\rightarrow V(M_+), \quad \p_0\mapsto \p_0\coprod M_+$$
is a homeomorphism since $M_+=\coprod_{d>0}M_d$ is a prime ideal of ${\sf Tot}(M_*)= \coprod_{d\geq 0}M_d$. We define $\Proj(M_*)$ as a topological space by
$$\Proj(M_*):=\{\p\in \Spec({\sf Tot}(M_*)) \ | \ \, M_+\setminus \p\not=\emptyset\}=\Spec({\sf Tot}(M_*))\setminus V(M_+).$$
Hence, $\Proj(M_*)$ is an open subset of $\Spec({\sf Tot}(M_*))$. Take an element $f\in {\sf Tot}(M_*)$. Then $D(f)\subseteq \Proj(M_*)$ if and only if $f\in M_+$. The collection of open sets $\{D(f)\}_{f\in M_+}$ is an open cover of $\Proj(M_*)$. 

We can regard $\Proj(M_*)$ as a monoid subscheme of $\Spec({\sf Tot}(M_*))$. But it also has another monoid scheme structure. We will need to introduce additional notation in order to describe it.

Let $f\in M_d$, $d\geq 1$ and denote by ${\sf Tot}(M_*)_{(f)}$ the subset of degree zero elements of ${\sf Tot}({M_*})_f$. That is
$${\sf Tot}(M_*)_{(f)}:=\left\{\frac{a}{f^k}\in {\sf Tot}({M_*})_f \ | \ \, a\in M_{kd}\right \}.$$
According to \cite{chww}, there exists a unique monoid scheme structure on $\Proj(M_*)$, for which the structure sheaf $\O_X$, $X=\Proj(M_*)$ is given by
$$\Gamma(D(f), \O_X)={\sf Tot}(M_*)_{(f)}, \ \ f\in M_+.$$
The stalk at $\p\in \Proj(M_*)$ is ${\sf Tot}(M_*)_{(\p)}$ and the restriction of the structural sheaf on $D(f)$, $f\in M_+$ is isomorphic to the affine monoid scheme $\Spec(M_{*(f)})$.

\subsection{The functor $\iota^*:\set_{{\rm gr}M_*}\rightarrow \Qc(\Proj(M_*))$}

Our next aim is to relate the category $\Qc(\Proj(M_*))$ of quasi-coherent sheaves over $X=\Proj(M_*)$ to the category $\set_{{\rm gr}M_*}$ of graded $M_*$-sets.

Assume $A_*$ is a graded $M_*$-set. There exist a unique sheaf on $\Proj(M_*)$, denoted by $\iota^*(A_*)$, such that
$$\Gamma(D(f),\iota^*(A_*))={{\sf Tot}(A_*)}_{(f)}$$
for each $f\in M_d$, $d\geq 1$.
It is clear from the description that $\iota^*(A_*)$ is a quasi-coherent sheaf. This yields the functor
$$\iota^*:\set_{{\rm gr}M_*}\rightarrow \Qc(\Proj(M_*)).$$
The functor $A_*\mapsto {{\sf Tot}(A_*)}_{(f)}$ respects colimits and finite limits. Hence, $\iota^*$ also preserves all colimits and finite limits.

\begin{Le}\label{722} Let $\alpha_*:A_*\rightarrow B_*$ be a morphism of graded $M_*$-sets. The following conditions are equivalent:
\begin{itemize}
\item [(i)] The morphism $\iota^*(\alpha_*):\iota(A_*)\rightarrow \iota^*(B_*)$ is a monomorphism.
\item [ii)] Let $f\in M_d$, $d\geq 1$ and $x,y\in A_k$ such that $\alpha_k(x)=\alpha_k(y)$. There exists $m\in \N$ such that $xf^m=yf^m$.
\end{itemize}
\end{Le}

\begin{proof} The family $\{D(f)\}_{f\in M_d, \ d\geq 1}$, forms an open cover. As such, $\iota^*(\alpha_*)$ is a monomorphism if and only if the restriction of $\iota^*(\alpha_*)$ on $D(f)$ is a monomorphism for each such $f$. The latter condition is equivalent to injectivity of the induced map $A_{*(f)}\rightarrow B_{*(f)}$, which happens if and only if ii) holds.
\end{proof}

\begin{Le} \label{723} Let $\alpha_*:A_*\rightarrow B_*$ be a morphism of graded $M_*$-sets. The following conditions are equivalent:
\begin{itemize}
\item [(i)] The morphism $\iota^*(\alpha_*):\iota^*(A_*)\rightarrow \iota^*(B_*)$ is an epimorphism.
\item [ii)] Let $f\in M_d$, $d\geq 1$ and $z\in B_k$. There exists $m\in \N$, such that $zf^m\in \im(\alpha_{mdk})$.
\end{itemize}
\end{Le}

\begin{proof} The family $\{D(f)\}_{f\in M_d, \ d\geq 1}$, forms an open cover of $\Proj(M_*)$. Hence, $\iota^*(\alpha_*)$ is an epimorphism if and only if the restriction of $\iota^*(\alpha_*)$ on $D(f)$ is an epimorphism for each $f\in M_d, d\geq 1$. The sheaves involved here are quasi-coherent and $D(f)$ is affine. In this case, this condition is equivalent to the surjectivity of the canonical map $A_{*(f)}\rightarrow B_{*(f)}$. The latter holds if and only if part ii) of this Lemma holds.
\end{proof}

\subsection{Some finiteness conditions}\label{74fin}

From this section onwards, up to Section \ref{connes}, we will assume that the positively graded monoid $M_*$ satisfies the following properties:
\begin{itemize}
\item[(i)] $M_0$ is a almost finitely generated monoid.
\item[(ii)] $M_1$ is a finitely generated $M_0$-set.
\item[(iii)] ${\sf Tot}(M_*)$ is generated by $M_0$ and $M_1$ as a monoid.
\end{itemize}

\begin{Le} Let $f_1,\cdots, f_s\in M_1$ generate $M_1$ as an $M_0$-set. The open subsets $D(f_1),\cdots, D(f_s)$ cover $\Proj(M_*)$.
\end{Le}

\begin{proof} Assume they do not. There exists a prime ideal $\p\in \Proj(M_*)$ such that $f_1, \cdots, f_s\in \p$. We can find an element $f=f_0f_1^{n_1}\cdots f_s^{n_s}\in M_+$ such that $f\not\in \p$, since $\p\in\Proj(M_*)$. Here, $f_0\in M_0$, $n_i\geq 0$ and $I=1,\cdots,s$. We see that $n_i>0$ for at least one $i>0$ since $f\not \in M_+$. It follows that $f\in (f_i)\subseteq \p$, which contradicts our assumptions on $f$. As usual, $(f_i)$ denotes the principal ideal generated by $f_i$.
\end{proof}

We will need the shifting functors on $\Qc(\Proj(M_*))$, which are defined as follows: Denote by $\O_X(n)$, $n\in \Z$, the sheaf $\iota^*(M_*(n))$, where $X=\Proj(M_*)$. Let $\mathcal{F}$ be  a quasi-coherent sheaf. We set
$$\mathcal{F}(n):=\mathcal{F}\otimes _{\mathcal O_X}\O_X(n).$$

\begin{Le} Let $X=\Proj(M_*)$. The sheaf $\O_X(n)$ is a line bundle for any $n\in \Z$.
\end{Le}

\begin{proof} We regard ${\sf Tot}(M_*)(n)_{(f)}$ as a ${\sf Tot}(M_*)_{(f)}$-set. Let $f\in M_1$. The restriction of $\O_X(n)$ on $D(f)$ is isomorphic to the quasi-coherent sheaf on $\Spec({\sf Tot}(M_*))_{(f)}$, associated to ${\sf Tot}(M_*)(n)_{(f)}$. Since $f$ is invertible in ${\sf Tot}(M_*)_{(f)}$, ${\sf Tot}(M_*)(n)_{(f)}$ is free of rank one over ${\sf Tot}(M_*)_{(f)}$ with generator $f^n$.
\end{proof}

\subsection{The geometric morphism $\iota:\Qc(\Proj(M_*)) \rightarrow \set_{{\rm gr}M_*}$}

Recall that $M_*$ satisfies the finiteness conditions listed at the beginning of Section \ref{74fin}. We have already mentioned that the functor
$$\iota^*:\set_{{\rm gr}M_*}\rightarrow \Qc(\Proj(M_*))$$
respects finite limits and all colimits. It has an adjoint $i_*$, which is given as follows: Take a quasi-coherent sheaf $\mathcal{F}$ on $X=\Proj(M_*)$ and define the graded set $\Gamma_*(\mathcal{F})$ by
$$\Gamma_n(\mathcal{F})=\Gamma(X,\mathcal{F}(n)), \ n\in \Z.$$
We will define an action of $M_*$ on $\Gamma_*(\mathcal{F})$. Take $m\in M_k$. Since $(M_*(k))_0=M_k$, the element $\frac{m}{1}$ can be considered as an element of $(M_*(k))_{*(f)}$ for any $f\in M_1$. This is the same as $\Gamma(D(f), \iota((M_*)(k))$. Such elements satisfy the gluing condition and hence, define a global element in 
$\Gamma(X, \O_X(k))$. It will still be denoted by $m$. Take an element $x\in \Gamma_i(\mathcal{F})$. By Lemma \ref{yonshe}, we have
$$\Hom_{\O_X}(\O_X, \mathcal{G})\simeq \Gamma(X, \mathcal{G})$$
for any sheaf of $\O_X$-sets $\mathcal{G}$. It follows that  $m$ and $x$ define morphisms of coherent sheaves $\O_X\rightarrow \O_X(k)$ and $\O_X\rightarrow \mathcal{F}(i)$ respectively. We take the tensor product to obtain
$$\O_X=\O_X\otimes_{\O_X}\O_X\rightarrow \O_X(k)\otimes_{\O_X}\mathcal{F}(i)=\mathcal{F}(k+i).$$
This is an element in $\Gamma_{k+i}(\mathcal{F})$, which we define to be $mx$. This gives us an action of $M_*$ on $\Gamma_*(\mathcal{F})$.

The assignment $\mathcal{F}\mapsto \Gamma_*(\mathcal{F})$ defines the functor $\iota_*: \Qc(\Proj(M_*)) \rightarrow \set_{{\rm gr}M_*}$.

\begin{Le} The functor $\iota^*$ is the left adjoint of $\iota_*$.
\end{Le}

\begin{proof} Let $A_*$ be a graded $M_*$-set and $\mathcal{F}$ be a quasi-coherent sheaf on $X=\Proj(M_*)$. We need to show that there is a bijective correspondence between morphisms $\alpha:A_*\rightarrow \Gamma_*(\mathcal{F})$ of graded $M_*$-sets and the morphisms of quasi-coherent sheaves of sets $\beta:\iota^*(A_*)\rightarrow \mathcal{F}$.

To define $\beta$ corresponding to $\alpha$, it suffices to define it on the affine subscheme $D(f)$ for each $f\in M_1$, and verify that these morphisms are compatible on the intersections. The category of quasi-coherent sheaves over $D(f)$ is equivalent to the category of $M_{*(f)}$-sets. Hence, we need to construct a morphism of $M_{*(f)}$-sets: $A_{*(f)}\rightarrow \Gamma(D(f), \mathcal{F})$.

Take an element $\frac{a}{f^k}$ of $A_{*(f)}$, where $a\in A_k$. Thus, $\alpha(a)\in \Gamma_k(\mathcal{F})$ is a morphism $\O\rightarrow \mathcal{F}(k)$. On the other hand, $\frac{1}{f^k}$ can be interpreted as a morphism $\mathcal {O}\rightarrow \O(-k)$ over $D(f)$. Hence, $\frac{1}{f^k}\otimes a$ defines a morphism $\O\rightarrow \mathcal{F}$, which is an element of $\Gamma(D(f), \mathcal{F})$.

To show that it is the restriction of a morphism $\beta(\frac{a}{f^k})$ on $D(f)$, one needs to check that they are compatible on intersection. To this end take $g\in M_1$. The image of $\frac{ag^k}{(fg)^k}$ is
$$\alpha(ag^k)\otimes \frac{1}{f^kg^k}=\alpha(a)\otimes \frac{g^k}{f^kg^k}=\frac{\alpha(a)}{f^k}.$$
These local data are therefore compatible and we get  $\beta:\iota^*(A_*)\rightarrow \mathcal{F}$. Thus, we have defined a map $\alpha\mapsto \beta$. It is a straightforward exercise to see that this is a bijection.
\end{proof}

We have constructed the geometric morphism $\iota:=(\iota^*,\iota_*):\Qc(\Proj(M_*)) \rightarrow \set_{{\rm gr}M_*}.$

\begin{Pro} \label{752} The geometric morphism $\iota:\Qc(\Proj(M_*)) \rightarrow \set_{{\rm gr}M_*}$ is a localisation of topoi.
\end{Pro} 

\begin{proof} It suffices to show that the canonical map $\beta:\iota^*(\Gamma_*(\mathcal{F}))\rightarrow \mathcal{F}$ is an isomorphism for any quasi-coherent sheaf $\mathcal{F}$ over $\Proj(M_*)$. This is equivalent to saying that there is an isomorphism
$$\Gamma_*(D(f),\mathcal{F})_{(f)}\simeq\Gamma(D(f),\mathcal{F})$$
for any $f\in M_1$.
Consider $f$ as an element of $\Gamma(X, \O_X(1))$. The desired isomorphism is a direct consequence of Lemma \ref{44123}, where the subset $X_f$ is exactly $D(f)$.
\end{proof}

\subsection{Points of the topos $\Qc(\Proj(M_*))$}

We will assume that our (positively) graded monoids satisfy the finiteness conditions listed in Section \ref{74fin}.

We would like to investigate the points of the topos $\Qc(\Proj(M_*))$. Recall that we have a topos point
$${\sf Stalk}(x):\Qc(\Proj(M_*))\rightarrow \set$$
for any $x\in X=\Proj(M_*)$, given by $\mathcal{F}\mapsto \mathcal{F}_x$.

We will show that any point of the topos $\Qc(\Proj(M_*))$ is of this form. We will need two auxiliary results for this.

Denote by $\Sigma$ the following class of morphisms in $\set_{{\rm gr}M_*}$: A morphism $\alpha:X_*\to Y_*$ of graded $M_*$-sets belongs to $\Sigma$, if $\iota(\alpha): \iota(X_*)\to \iota(Y_*)$ is an isomorphism.

\begin{Le} \label{762} Let $\p$ be a prime ideal of $\,{\sf Tot}(M_*)$. If the functor $A_*\mapsto (A_{*}(i))_{(\p)}$ takes morphisms from $\Sigma$ to isomorphisms, then $\p\in \Proj(M_*)$.
\end{Le}

\begin{proof} Without loss of generality, we may assume that $i=0$. Observe that the natural inclusion $M_+\subseteq {\sf Tot}(M_*)$ belongs to $\Sigma$. Hence, $(M_+)_{(\p)}\rightarrow M_{*(\p)}$ is an isomorphism, thanks to Lemmas \ref{722} and \ref{723}. It follows that $1$ is in the image of this map. We have $1=\frac{m}{s}$, where $s\nin \p$ and $m\in M_+$. Thus, $|s|=|m|>0$ and $s\in M_+$. It follows that $M_+\setminus \p\not=\emptyset$ and $\p\in \Proj(M_*)$. 
\end{proof}

\begin{Le} \label{763} Let $\p\in \Proj(M_*)$. The functors $A_*\mapsto(A_{*}(i))_{(\p)}$ and $A_*\mapsto(A_{*})_{(\p)}$ are isomorphic for all $i\in \Z$.
\end{Le}

\begin{proof} By definition, $(A_{*})_{(\p)}$ (resp. $(A_{*}(i))_{(\p)})$ is the degree zero (resp. degree $i$-th)  part of ${\sf Tot}(A_*)_\p$. It follows from our finiteness assumption that there exists an element $f\in M_1$, such that $f\nin \p$. Hence, $l_f:{\sf Tot}(A_*)_\p \rightarrow {\sf Tot}(A_*)_\p$, $x\mapsto fx$ is an isomorphism which raises the degree by one. The restriction of the $i$-th iteration of this isomorphism on $(A_{*})_{(\p)}$ yields the isomorphism in question.
\end{proof}

Now we are in the position to state and prove the following result:

\begin{Th} Let $M_*$ be a positively graded commutative monoid satisfying the finiteness conditions listed in Section \ref{74fin} and let $P^*: \Qc(\Proj(M_*))\rightarrow \set$ be the inverse image of a topos point $\Qc(\Proj(M_*))$. There exists an isomorphism of points $P^*\simeq {\sf Stalk}(x)$ for some element $x\in\Proj(M_*)$.
\end{Th}

\begin{proof} By Lemma \ref{points_after_loc} and Proposition \ref{752}, the isomorphism classes of topos points $\Qc(\Proj(M_*))$ correspond to the isomorphism classes of points $P^*: \set_{{\rm gr}M_*}\rightarrow \set$, for which $P^*(\alpha)$ is a bijection for any $\alpha\in \Sigma$. Here, $\Sigma$ is the same as in Section \ref{74fin}.

On the other hand, Theorem \ref{pgrprime} gives a description of all the points of the topos $\set_{{\rm gr}M_*}$. The result follows from Lemmas \ref{762} and \ref{763}.
\end{proof}

\begin{Co}\label{7540704} Let $M_*$ be a positively graded commutative monoid satisfying the finiteness conditions listed in Section \ref{74fin} and let $U$ be an open monoid subscheme of $X=\Proj(M_*)$. Any $P^*: \Qc(U)\rightarrow \set$ is isomorphic to a point of the form ${\sf Stalk}(x)$, $x\in U$.
\end{Co}

We call a monoid scheme $X$ \emph{quasi-projective} if it is isomorphic to an open monoid subscheme of $\Proj(M_*)$. Here, $M_*$ is a non-negative graded monoid satisfying the finiteness conditions listed in Section \ref{74fin}.

\begin{Th}\label{7550804} Let $X$ and $Y$ be quasi-projective monoid schemes such that the categories $\Qc(X)$ and $\Qc(Y)$ are equivalent. Then $X$ and $Y$ are isomorphic. 
\end{Th}

\begin{proof} The monoid scheme structure of $X$ is determined by the underlying ordered set and stalk functors $x\mapsto \O_{X,x}$, as $X$ it is of finite type \cite[Proposition 2.11]{chww}, \cite[Proposition 2.3]{p2}. It suffices to show that the category of quasi-coherent sheaves determines $X$ as a poset and the stalk functors. 

Let us consider $\pts(\Qc(X))$, the equivalence classes of topos points of $\Qc(X)$. Define an ordering on these equivalence classes by saying that $P\leq Q$ if $P\in\Delta(Q)$. Here, $P,Q\in\pts(\Qc(X))$ are topos points. Corollary \ref{x5} shows that this is a partial ordering in our case. Further, it agrees with the ordering of the underlying set of $X$. We have also an isomorphism of monoids $\O_{X,P}\simeq \End(P)$ and the result follows.
\end{proof}

\section{Infinitely generated commutative monoids}\label{connes}

We will now investigate the case of non-finitely commutative generated monoids. We will give a description of the topos points of free commutative monoids, and show that in this case, the prime ideals are no longer enough to classify them all.

\subsection{Cancellative monoids}

Recall that a monoid $M$ is called cancellative if for any elements $a,b,c$ of $M$, $ab=ac$ implies that $b=c$. Condition $(F2)$ in Lemma \ref{Fcond} can then be simplified as saying that if $m_1a=m_2a$, then $m_1=m_2$.

\begin{Le} Let $M$ be a cancellative monoid and $A$ be a filtered $M$-set.
\begin{itemize}
\item[i)] Let $a,b\in A$ and $m\in M$. If $ma=mb$, then $a=b$.
\item[ii)] Let $Ma=Mb$, where $a,b\in A$. There exists an invertible element $m\in M$ such that $a=mb$.
\end{itemize}
\end{Le}

\begin{proof} i) There are elements $c\in A$, $m_1,m_2\in M$ by (F3), such that $m_1c=a$ and $m_2c=b$. So, $m_1mc=m_2mc$. Hence, $m_1m=m_2m$ and $m_1=m_2$.

ii) We have $a=mb$ and $b=na$ for some $m,n\in M$. As such, $a=mna$ and $mn=1$. This finishes the proof.
\end{proof}

Let $M$ be a commutative monoid. The universal (or Grothendieck-) group of $M$ will be denoted by $\Mg$. Recall that it consists of elements of the form $\frac{a}{b}$, $a,b\in M$. Two elements $\frac{a}{b}$ and $\frac{c}{d}$ are equivalent in $\Mg$ if there exists an element $s\in M$, such that $sad=scb$.

\begin{Le} \label{f411} Let $M$  be a cancellative monoid.
\begin{itemize}
\item[i)] Any filtered $M$-set is isomorphic to an $M$-subset of $\Mg$.
\item[ii)] Assume $A$ and $B$ are filtered $M$-subsets of $\Mg$ and $\alpha:A\rightarrow B$ is an $M$-set homomorphism. There exists $b\in\Mg$, such that $\alpha(a)=ab$ for all $a\in A$. 
\end{itemize}
\end{Le}

\begin{proof} i) This is a direct consequence of Corollary \ref{s1} and Example \ref{t-233}.

ii) It follows from the proof of Corollary \ref{s1} and Example \ref{t-233} that we have a commutative diagram
$$\xymatrix{ A\ar[r]^{\alpha}\ar[d]_{i_A} & B\ar[d]^{i_B} \\
			 A\otimes _M\Mg\ar[r]^{\alpha_*}\ar[d]_{\simeq} & B\otimes _M\Mg\ar[d]^{\simeq}\\
			 \Mg\ar[r]_{\beta} & \Mg. }$$
Here, $i_A$ and $i_B$ are injective maps. We have $\beta(ax)=a\beta(x)$ for all $a,x\in \Mg$ since $\beta$ is a $\Mg$-map. Take $x=1$ to obtain $\beta(a)=ab$, where $b=\beta(1)$.
\end{proof}

\begin{Co} Let $M$ be a cancellative monoid and $A$ and $B$ filtered $M$-sets. Any morphism of $M$-sets $\alpha:A\rightarrow B$ is injective.
\end{Co}

\begin{Pro} Let $M$ be a commutative monoid and $P$ a topos point of $\set_M$. Assume $M$ is cancellative or almost finitely generated. The monoid $\End(P)$ of endofunctors of $P$ is commutative.
\end{Pro}

\begin{proof} Let $A$ be a filtered $M$-set corresponding to the point $P$. We know that there is an isomorphism of monoids $\End_{\set_M}(A)\simeq\End_M(P)$.

Assume $M$ is cancellative and let $\phi,\psi\in \End_{\set_M}(A)$. By Lemma \ref{f411}, we can assume that $A$ is an $M$-subset of $\Mg$ and that $\phi(a)=s_\phi\cdot a$, $\psi(a)=s_\psi\cdot a$, for all $a\in A$. Here, $s_\psi$ and $s_\phi$ are appropriate elements in $\Mg$. It follows that
$$\phi\psi(a)=s_\psi s_\phi\cdot a=\psi\phi(a).$$
Hence, $\End_{\set_M}(A)$ is commutative in this case.

If $M$ is almost finitely generated, we have $A\simeq M_\p$ for a prime ideal $\p$ by Theorem \ref{thm}. One easily sees that $\End_{\set_M}(A)\simeq M_\p$ as monoids. This implies the result.
\end{proof}

\subsection{Free Monoids}

In what follows, $M$ will denote a free commutative monoid with a basis set $\cX$. The aim is to investigate $\Pts(\set_M)$, the points of the topos of $M$-sets for arbitrary $\cX$. We will describe the isomorphism classes of $\Pts(\set_M)$ in terms of equivalence classes of certain functions $\cX\rightarrow \Z\coprod \{-\infty\}$. Our result is equivalent to the one obtained in \cite{cc1} when $\cX$ is countably infinite. Our methods and description, however, are rather different.

Recall that we denote the universal group of $M$ by $\Mg$. In our case, it is the free abelian group with basis $\cX$. Thus, any element $x\in\Mg$ has a unique presentation as
$$x=\prod_{\p\in \cX}\p^{\v_{\p(x)}}.$$
Here, $\v_\p(x)\in\Z$ and $\v_\p(x)=0$ for all but a finite number of $\p$. It is also clear that $x\in M$ if and only if $\v_\p(x)\geq 0$ for all $\p$.

\subsubsection{The set $ \Sigma_*(M)$}

Denote by $\Sigma(M)$ the set of all functions
$$f:\cX\rightarrow \mathbb{Z}\coprod \{-\infty\}$$
for which $f(\p)\leq 0$ for all but finitely many $\p\in \cX$. We define a congruence relation on $\Sigma (M)$ where $f\sim g$ provided: $f(\p)=-\infty$ if and only if $g(\p)=-\infty$, and $f(\p)=g(\p)$ for all but finitely many $\p\in \cX$. In other words, if the sum of all their differences $\sum_{\p\in\cX}f(\p)-g(\p)$ is finite.

Denote by $\Sigma_*(M)$ the set of equivalence classes $\Sigma(M)/\sim$ under this equivalency. The obvious addition in $\mathbb{Z}\coprod \{-\infty\}$ induces a monoid structure on $\Sigma_*(M)$.

\begin{Th}\label{m412} Let $M$ be a free monoid. There is an isomorphism of monoids
$${\sf F}_M\simeq \Sigma_*(M).$$
\end{Th}

Recall that we denoted by ${\sf F}_M$ the isomorphism classes of the topos points of $M$-sets. The proof is based on several auxiliary results and is given in Subsection \ref{proof of m412}.

\subsubsection{Auxiliary results}

Let $A$ be a filtered $M$-subset of $\Mg$ and $\p\in \cX$. That is to say, $A$ is a subset of $\Mg$, which is additionally filtered as an $M$-set. We know by part i) of Lemma \ref{f411} that all $M$-sets are of this form. Denote
$$\v_\p(A)={\sf inf}\{\v_\p(a) \ | \ a\in A\}.$$
We have $\v_\p(A)\not =+\infty$ since $A\not=\emptyset$. Thus, $\v_\p(A)\in \mathbb{Z}\coprod \{-\infty\}$. We obtain the function
$$\v(A):\cX\rightarrow \mathbb{Z}\coprod \{-\infty\},\ \ \ \p\mapsto \v_\p(A)$$
by varying $\p\in \cX$.

\begin{Pro}\label{m414} Let $M$ be a free commutative monoid on a basis set $\cX$.
\begin{itemize}
	\item[i)] Let $A$ be a filtered $M$-subset of $\Mg$. Then $\v(A)\in \Sigma(M)$ and
	$$A=\{x\in \Mg \ | \ \v_\p(x)\geq \v_\p(A) \quad {\rm for \ all} \quad \p\in \cX\}.$$
	\item[ii)] Let $f\in \Sigma(M)$ and set
	$$A(f)=\{x\in \Mg \ | \ \v_\p(x)\geq f(\p)\ {\rm for \ all} \ \p\in \cX\}.$$
	Then $A(f)$ is a filtered $M$-subset of $\Mg$.
\end{itemize}
\end{Pro}

\begin{proof} i) Assume $\v(A)\nin\Sigma(M)$. We have $\v_\p(A)>0$ for infinitely many $\p\in \cX$. Let $a\in A$. We have $\v_\p(a)\geq \v_\p(A)>0$ for infinitely many $\p$. This contradicts the fact that $\v_\p(a)=0$ for all but finitely many $\p$ since $A\not=\emptyset$. Hence, $\v(A)\in\Sigma(M)$.

For the next assertion, observe that
$$A\subseteq \{x\in \Mg \ | \ \v_\p(x)\geq \v_\p(A) \quad {\rm for \ all} \quad \p\in \cX\}$$
by the definition of $\v(A)$. To see the other side, take an element $x\in \Mg$, such that $\v_\p(x)\geq \v_\p(A)$ for all $\p\in \cX$. We have to show that $x\in A$. We can choose elements $a_\p\in A$, $\p\in \cX$ for which
\begin{equation}\label{vpx}\v_\p(x)\geq \v_\p(a_\p),
\end{equation}
by our assumption on $x$. Moreover, we can assume $\v_\p(a_\p)\leq 0$ if $\v_\p(A)\leq 0$. Define
$$\sigma(x)=\{\p\in \cX \ | \ \v_\p(x)\not=0\}.$$
Observe that $\sigma(x)$ is a finite set. It follows from the filteredness axiom (F3) that there exists an element $a\in A$ and a finite set $\{m_\p\in M \ | \ \p\in \sigma(x)\}$, such that $a_\p=m_\p a$ for all $\p\in \sigma(x)$. Likewise, we introduce the finite set
$$\pi(a)=\{\p\in \cX \ | \ \p\nin \sigma(x) \ \ {\rm and} \ \ \v_\p(a)>0\}.$$
We use axiom (F3) once again to deduce that there exist elements $b\in A$, $m\in M$ and a finite set $\{m_\r\in M \ | \ \r\in \pi(a)\}$, such that $a=mb$ and $a_\r=m_\r b$ for all $\r\in \pi(a)$.

We claim that
$$\v_\p(x)\geq \v_\p(b) \ \  \text{ for \ all } \ \ \p\in \cX.$$ 
To check the claim, we first observe that 
$$\v_\p(a)=\v_\p(m)+\v_\p(b)\geq \v_\p(b) \ \ \text{ for all } \ \ \p\in \cX.$$
This is true since $m\in M$ and so, $\v_\p(m)\geq 0$. There are three cases to be considered.

\begin{itemize}
\item[Case 1:] $\p\in \sigma(x)$.  By Equation (\ref{vpx}), we have
$$\v_\p(x)\geq \v_\p(a_\p)=\v_\p(m_\p)+\v_\p(a)\geq \v_\p(a)\geq \v_\p(b).$$
\item[Case 2:] $\p\nin \sigma(x)$ and  $\p\nin \pi(a)$.  We have $\v_\p(x)=0$ and $\v_\p(a)\leq 0$ by assumption. This implies that
$$\v_\p(x)=0\geq \v_\p(a)\geq \v_\p(b).$$ 
\item[Case 3:] $\p\in \pi(a)$. Equation (\ref{vpx}) and our assumption implies that we have
$$0=\v_\p(x)\geq \v_\p(a_\p)=\v_\p(m_\p)+\v_\p(b)\geq \v_\p(b).$$
\end{itemize}
This proves the claim. Write $x=nb$, where $n=xb^{-1}\in \Mg$. We have
$$\v_\p(n)=\v_\p(x)-\v_\p(b)\geq 0, \  \p\in \cX.$$
It follows that $n\in M$. We see that $x\in A$ since $x=nb$, $b\in A$, $n\in M$ and $A$ is an $M$-subset of $ \Mg$.

ii) If $x\in A(f)$ and $m\in M$, then $\v_\p(mx)=\v_\p(m)+\v_\p(x)\geq \v_\p(x)\geq f(\p)$ and $mx\in A(f)$. It follows that $A(f)$ is an $M$-subset of $\Mg$. To see that it is filtered, we have to check axioms (F1)-(F3). Axiom (F2) holds automatically as $A(f)\subseteq \Mg$ and $\Mg$ is a group and in particular, cancellative.

To see that $A(f)$ is non-empty, take an element $a\in \Mg$ such that
$$\v_\p(a)=\begin{cases} 0, \ {\rm if} \ f(\p)\leq 0,\\
f(\p), \ {\rm if} \ f(\p)>0. \end{cases}$$
Such an element exists because $f$ takes positive values for only a finite number of $\p\in \cX$. It is clear that $a\in A(f)$. This implies $(F1)$.

Take two elements $b,c\in A(f)$. We choose a finite subset $S\subseteq \cX(M)$, such that $f(\p)\leq 0$ and $\v_\p(b)=0=\v_\p(c)$ for all elements $\p\nin S$. Take $d\in \Mg$ such that
$$\v_\p(d)=\begin{cases} 0, \ {\rm if} \ \p\nin S\\
f(\p), \ {\rm if} \ \p\in S. \end{cases}$$
We have $d\in A(f)$. Define $m=bd^{-1}$ and $n=cd^{-1}$. Then $\v_\p(m), \v_\p(n)\geq 0$ for all $\p\in \cX$ and thus, $m,n\in M$. Since $b=md, c=nd$, (F3) follows as well.
\end{proof}

We have shown in the above proposition, in conjunction with part i) of Lemma \ref{f411}, that elements of $\Sigma(M)$ define filtered $M$-sets, and that every filtered $M$-set is of that form. We will now investigate when two two functions $f,g\in\Sigma(M)$ define isomorphic $M$-set.

\begin{Le}\label{m415a} Let $f,g\in \Sigma(M)$ and assume
$$\Hom_{\set_M}(A(f),A(g))\not=\emptyset.$$
Then $f(\p)=-\infty$ implies $g(\p)=-\infty$. Moreover, $f(\p)\geq g(\p)$ for all but a finite number of $\p\in \cX(M)$.
\end{Le}

\begin{proof} Let $\psi:A(f)\rightarrow A(g)$ be an $M$-set homomorphism. There exists an element $a\in \Mg$, such that $\psi(x)=ax$, $x\in A(f)$. This follows directly from Lemma \ref{f411}. Hence, $\v_\p(a)+f(\p)\geq g(\p)$ for all $\p\in {\cX}$ and the result follows.
\end{proof}

\begin{Le}\label{m415b} Let $f,g\in \Sigma(M)$. The isomorphism of $M$-sets
$$\alpha:\Mg\otimes _M\Mg\rightarrow \Mg,$$
given by $\alpha(a\otimes b)=ab$, restricts to an isomorphism of $M$-sets
$$A(f)\otimes_M A(g)\rightarrow A(f+g).$$
\end{Le}

\begin{proof} Take  $a\in A(f)$ and $b\in A(g)$. We have $\v_\p(ab)=\v_\p(a)+\v_\p(b)\geq f(\p)+g(\p)$. This shows that the image of $A(f)\otimes_MA(g)$ under $\alpha$ lies in $A(f+g)$. We get a commutative diagram
$$\xymatrix{ A(f)\otimes_MA(g)\ar[r]^\beta\ar[d] & A(f+g)\ar[d]\\
			 \Mg\otimes_M \Mg\ar[r]_{\ \ \alpha} & \Mg. }$$
The vertical arrows are injective and $\beta$ is a restriction of $\alpha$. Since $\alpha$ is bijective, $\beta$ is bijective if and only if it is surjective. Let $c\in A(f+g)$ and define
$$S=\{\p\in \cX \ | \ f(\p)\leq 0, g(\p)\leq 0 \ {\rm and} \ \v_p(c)=0\}.$$
We can choose $a\in\Mg$, such that
$$\v_\p(a)=\begin{cases} f(\p), \ {\rm if} \ \p\nin S\\
						 0, \ {\rm if} \ \p\in S \end{cases}$$
since $\cX\setminus S$ is finite. We have $a\in A(f)$ and define $b=ca^{-1}\in \Mg$. We claim that $b\in A(g)$. Assume $\p\nin S$. We have
$$\v_\p(b)=\v_\p(c)-\v_\p(a)\geq f(\p)+g(\p)-f(\p)=g(\p).$$
On the other hand, if $\p\in S$, one has
$$\v_\p(b)=\v_\p(c)-\v_\p(a)=0 \geq g(\p).$$
This shows that $\beta(a\otimes b)=c$ and hence, $\beta$ is surjective. This finishes the proof.
\end{proof}

\subsubsection{Proof of Theorem \ref{m412}}\label{proof of m412}

We are now in position to give the proof of the main theorem of this section.

\begin{proof} We know by part ii) of Proposition \ref{m414} that there exists a filtered $M$-set $A(f)$ for every $f\in \Sigma(M)$. This association induces a map $\Sigma(M)\rightarrow {\sf F}_M$. This is surjective by part i) of Proposition \ref{m414}. Assume $f\sim g$ and define
$$S=\{\p\in \cX \ | \ f(\p)\not=g(\p)\}.$$
There exist an element $b\in \Mg$ such that
$$\v_\p(b)=\begin{cases} 0, \ {\rm if} \ \p\not\in S\\
\v_\p(g)-\v_\p(f) , \ {\rm if} \ \p\in S,\end{cases}$$
since $S$ is a finite set. As such, $\psi:A(f)\rightarrow A(g)$ is an isomorphism of $M$-sets. Here,
$$\psi(a)=ab, \ a\in A(f).$$ 
This shows that the above map $\Sigma(M)\rightarrow {\sf F}_M$ factors through $\Sigma_*(M)\rightarrow {\sf F}_M$, which is naturally surjective as well.

Take $f,g\in \Sigma(M)$ and assume that $A(f)$ and $A(g)$ are isomorphic $M$-sets. It follows from Lemma \ref{m415a} that $f\sim g$. This proves bijectivity. It respects the monoid structure according to Lemma \ref{m415b}.
\end{proof}

\subsection{On the maps $\alpha_M$ and $\gamma_M$}\label{algam}

Let us return (see Proposition \ref{55125}) to the injective homomorphisms of commutative monoids
$$M^{sl}\xrightarrow{\alpha_M}\Spec(M)\xrightarrow{\gamma_M}F_M.$$
We assume that $M$ is a free monoid generated by the infinite set $\cX$.

To see that $\alpha_M$ is not an isomorphism, we first observe that ${\sf Card}(M)={\sf Card}(\cX)={\sf Card}(M^{sl})$. As previously shown \cite[Lemma 2.1]{p1}, we have an isomorphism 
$$\Spec(M)\simeq \Hom(M,\I).$$
Here, $\I=\{0,1\}$ with obvious multiplication. It follows that $|\Spec(M)|=2^{\cX}$ since $M$ is free with basis $\cX$. As such, ${\sf Card}(\Spec(M))\gneqq{\sf Card}(M^{sl})$ is strictly bigger and $\alpha_M$ is not an isomorphism.

\begin{Le} The homomorphism
$$2^{\cX}\simeq \Spec(M)\xrightarrow{\gamma_M} {\sf F}_M\simeq \Sigma_*(M)$$ 
has the following description, according to Theorem \ref{m412}: Let $T$ be a subset of $\cX(M)$. Then $\gamma_M(T)\in\Sigma_*(M)$ is represented by the function $\gamma(T):\cX(M)\rightarrow \Z\coprod{-\infty}$, given by
$$\gamma(T)(\p)=\begin{cases}-\infty, \ {\rm if} \ \ \p\nin T\\ 0, \ {\rm if} \ \p\in T. \end{cases}$$
\end{Le}

\begin{proof} It is clear that the prime ideal $\p_T$ corresponding to $T$ is given by
$$\p_T=\{x\in M \ | \ \v_\p(x)>0 \ {\rm for \ all} \ \p\in T\}.$$
Therefore, the localisation of $M$ at the prime ideal $\p_T$ is 
$$\{x\in \Mg \ | \ \v_\p(x)\geq 0 \ {\rm for \ all} \ \p\in T\}=A(\gamma(T)).$$
\end{proof}

It is clear from the description that $\gamma_M$ is not an isomorphism for infinite $\cX$: Indeed, let $f\in \Sigma$ be any function for which $-\infty<f(\p)<0$ for infinitely many $\p$. It defines an element in $\Sigma_*(M)$, which is not in the image of $\gamma_M$.

\begin{center}

\end{center}
\end{document}